\documentclass[11pt,onecolumn]{IEEEtran}
\usepackage{amsthm,amssymb,graphicx,multirow,amsmath,color,amsfonts}
\usepackage[noadjust]{cite}
\usepackage{tikz}
\usetikzlibrary{arrows,calc}		
\usepackage{bbm} 
\usepackage{pdfpages}
\usepackage{comment}
\usepackage{dsfont}
\usepackage{mathtools}
\usepackage{hyperref}

\newcommand\numberthis{\addtocounter{equation}{1}\tag{\theequation}}

\makeatletter
\renewcommand{\env@cases}[1][@{}l@{\quad}l@{}]{%
  \let\@ifnextchar\new@ifnextchar
  \left\lbrace
  \def\arraystretch{1.2}%
  \array{#1}%
}
\makeatother

\def\nbx{{\mathbf{x}}}
\def\nby{{\mathbf{y}}}

\def\nb0{{\mathbf{0}}}
\def\nb1{{\mathbf{1}}}

\def\ncalB{{\mathcal{B}}}

\def\ncalI{{\mathcal{I}}}

\newtheorem{lemma}{Lemma}
\newtheorem{thm}{Theorem}
\newtheorem{definition}{Definition}

\newtheorem{theorem}{Theorem}

\newtheorem{cor}{Corollary}

\newtheorem{remark}{Remark}

\def\E{\mathbb{E}}

\def\R{\mathbb{R}}

\newcommand{\mytheta}{\boldsymbol{\theta}}
\newcommand{\myalpha}{\boldsymbol{\alpha}}

\allowdisplaybreaks 

\begin{document}
\graphicspath{{./Figures/}}
\title{Distance from the Nucleus to a Uniformly \newline Random Point in the 0-cell and the Typical Cell of the Poisson-Voronoi Tessellation}
\author{
Praful D. Mankar, Priyabrata Parida, Harpreet S. Dhillon, Martin Haenggi
\thanks{P. D. Mankar, P. Parida and H. S. Dhillon are with Wireless@VT, Bradley Department of Electrical and Computer Engineering, Virginia Tech, Blacksburg, VA. Email: \{prafuldm, pparida, hdhillon\}@vt.edu. M. Haenggi is with the Department of Electrical Engineering, University of Notre Dame, Notre Dame, 
IN, Email: mhaenggi@nd.edu. This work was supported by the US National Science Foundation under Grant ECCS-1731711. 
}
}

\maketitle

\begin{abstract}
Consider the distances $\tilde{R}_o$  and $R_o$ from the nucleus to a uniformly random point in the 0-cell and the typical cell, respectively, of the $d$-dimensional Poisson-Voronoi (PV) tessellation. The main objective of this paper is to characterize the exact distributions of $\tilde{R}_o$ and $R_o$. First, using the well-known relationship between the 0-cell and the typical cell, we show that the random variable $\tilde{R}_o$ is equivalent in distribution to the contact distance of the Poisson point process. 
Next, we derive a multi-integral expression for the exact distribution of $R_o$. Further, we derive a closed-form approximate expression for the distribution of $R_o$, which is the contact distribution with a mean corrected by a factor equal to the ratio of the mean volumes of the 0-cell and the typical cell. An additional outcome of our analysis is a direct proof of the well-known spherical property of the PV cells having a large inball.
\end{abstract}
\begin{keywords}
Poisson point process, Poisson-Voronoi tessellation, typical cell, 0-cell, distance distribution.
\end{keywords}        
\section{Introduction}\label{sec:Intro} 

The Poisson point process (PPP) has found many applications in science and engineering due to its useful mathematical properties. Several of these applications specifically focus on the Poisson-Voronoi (PV) tessellation \cite{moller1989random}, which partitions space into disjoint {\em cells} whose nuclei are the points of the PPP. There is a rich literature focused on characterizing the statistical properties of the PV tessellation, such as the distributions of the contact and chord lengths \cite{muche1992contact}, the distributions of the radii of the circumcircle and the incircle of the 0-cell and the typical cell \cite{Calka2002}, the distribution of the number of edges of the typical cell \cite{calka2003explicit}, the limiting shape of the 0-cell and the typical cell \cite{hug2004large}, and the relationship between the 0-cell and the typical cell \cite{MECKE1999}. 
However, it is quite surprising to note that the distributions of the distances from the nucleus to uniformly random points in the 0-cell and the typical cell of the $d$-dimensional PV tessellation have not  yet been investigated, which is the main goal of this paper.

The motivation behind our investigation comes from wireless networks, where the PPP has been extensively used to model the locations of cell towers (also called base stations) in cellular networks such that the service region of each cell tower is simply the PV cell with the corresponding cell tower at its nucleus \cite{BacBla2009,AndBacJ2011,DhiGanJ2012,Haenggi2013,blaszczyszyn_haenggi_keeler_mukherjee_2018}. 
If one assumes mobile users to be distributed uniformly at random in the service region of each cell tower (a popular model used by the wireless networks community), one of the crucial steps towards the performance characterization of this network is to understand the distribution of the distance between a mobile user and its associated cell tower. In the PV tessellation, this corresponds to the distribution of the distance of the nucleus of a PV cell to a point chosen uniformly at random from that cell \cite{Haenggi2017,Praful_TypicalCell_letter}. Note that while it is sufficient to focus on the $2$-dimensional case from the wireless networks perspective, all the mathematical results presented in this paper are for the general $d$-dimensional case. With this brief introduction, we now formally define the problem of interest for this paper. 

Let $\Phi\triangleq\{\textbf{x}_1,\textbf{x}_2,\dots\}$ be a homogeneous PPP with intensity $\lambda$ on $\mathbb{R}^d$. The PV cell with the nucleus at $\textbf{x}\in\Phi$ is defined as
\begin{equation}
V_\mathbf{x}=\{\mathbf{y}\in\mathbb{R}^d\mid \|\mathbf{y}-\mathbf{x}\|\leq\|\mathbf{x'}-\mathbf{y}\|, ~\forall \mathbf{x'}\in\Phi\},~~~~ \mathbf{x}\in\Phi.
\label{eq:PV_Cell_Definition}
\end{equation}
The set $\{V_\textbf{x}\}_{\textbf{x}\in\Phi}$ is known as the {\em PV tessellation}. For any (deterministic) $\textbf{y}\in\mathbb{R}^d$, almost surely there exists a unique nucleus $\mathbf{x}\in\Phi$  such that $\textbf{y}\in V_\textbf{x}$ \cite{Moller2012lectures}.  The PV cell containing the origin $o$  is called the {\em 0-cell} and is denoted by $\tilde{V}_o$. 
The statistical properties of the {\em typical cell} can be characterized using Palm theory, which formalizes the notion of conditioning on the presence of a point of a point process at a specific location. Since by Slivnyak's theorem, conditioning on a point is the same as adding a point to a PPP, we consider that the nucleus of the typical cell of the point process $\Phi\cup\{o\}$ is $o$, which is given by
\begin{equation}
V_o=\{\mathbf{y}\in\mathbb{R}^d\mid \|\mathbf{y}\|\leq\|\mathbf{x}-\mathbf{y}\|, ~\forall \mathbf{x}\in\Phi\}.
\end{equation}
Now, using $\tilde{V}_o$ and $V_o$, we define the main random variables of interest for this paper.
\begin{definition} 
Let $\tilde{R}_o$ denote the distance from the nucleus to a uniformly random point in the 0-cell $\tilde{V}_o$.
\end{definition}
\begin{definition} 
Let $R_o$ denote the distance from the nucleus to a uniformly random point in the typical cell $V_o$.
\end{definition}

We focus on the statistical characterization of $R_o$ and $\tilde{R}_o$ for  the PPP with intensity $\lambda$. We derive the cumulative distribution function (CDF) of $\tilde{R}_o$ and $R_o$ in Sections \ref{sec:CroftonCell} and \ref{sec:TypicalCell}, respectively. In Section \ref{sec:CroftonCell}, a closed-form expression for the exact CDF of $\tilde{R}_o$ is derived based on the formula on the relationship between the 0-cell and the typical cell given in \cite{BacBla2009,MECKE1999}. It is well-known that the statistical properties of $R_o$ are hard to characterize for the case of $d>1$. Before going into the $d>1$ case, we discuss the case of $d=1$ in Section \ref{subsec:1D} for which the distribution of $R_o$ is far easier to characterize. In Section \ref{sec:GeneralCase_CDF_R},  we present an analytical approach to derive the distribution of $R_o$ for the $d>1$ case based on the analysis of the temporal evolution of the PV structure presented in \cite{Pineda}. 
We also approximate the CDF of $R_o$ using a simple expression in Section \ref{sec:Approximate_CDF}. Therein, we also characterize the distribution of $R_o$ as $d$ tends to infinity. In addition, based on the formulation developed in Section \ref{sec:TypicalCell}, we provide a simpler proof for the well-known spherical nature of large PV cells in Section \ref{sec:LimitingShape}.

\section{Distribution of $\tilde{R}_o$}
\label{sec:CroftonCell}
In this section, we derive a closed-form expression for the CDF of  the distance from the nucleus to uniformly random point in the 0-cell $\tilde{V}_o$. It is well-known that the expected volume of the 0-cell is greater than the expected volume of the typical cell. In fact, all the moments of the volume of the 0-cell are known to be greater than the moments of the volume of the typical  cell \cite{MECKE1999}. This is quite intuitive as the origin (or, for that matter, any {\em fixed} point) is more likely to lie in a bigger cell. 

Before presenting the CDF of $\tilde{R}_o$, we state the relationship of the distributions of the 0-cell and the typical cell  from \cite[Corollary 4.2.4]{BacBla2009} as 
\begin{equation}
\mathbb{E}[f(\tilde{V}_o)]=\lambda\mathbb{E}^{o}[\upsilon_d(V_o)f(V_o)],
\label{eq:Relation_Vo_V}
\end{equation}
where $\upsilon_d$ is the Lebesgue measure in $d$-dimensions, $\E^{o}$ is the expectation with respect to the Palm distribution, and $f$ is any translation-invariant non-negative function on compact sets. 
We will use this expression along with an appropriately chosen function $f$ to derive the CDF of $\tilde{R}_o$ in Theorem~\ref{thm:CDF_R_Crofton}.
Let $\mathcal{B}_r(\mathbf{x})$ represent the $d$-dimensional ball of radius $r$ centered at $\mathbf{x}$. 
Let $X$ be a random set in $\mathbb{R}^d$. Using the results of \cite{robbins1944} and \cite{robbins1945}, the $n$-th moment of the volume of $X$ can be evaluated as
\begin{equation}
\mathbb{E}[\upsilon_d(X)^n]=\int_{\mathbb{R}^d}\dots\int_{\mathbb{R}^d}\mathbb{P}(x_1,\dots, x_n\in X){\rm d}x_1\dots{\rm d}x_n.
\label{eq:Moment_Random_Set}
\end{equation} 
Next, we restate a useful result from \cite[Lemma 4.2]{alishahi2008} on the mean volume of $\ncalB_r(o)\cap V_o$, which directly follows from \eqref{eq:Moment_Random_Set}.
\begin{lemma}
\label{lemma:Mean_IntBall_PVCell}
For the homogeneous PPP with intensity $\lambda$ on $\mathbb{R}^d$, the mean volume of the intersection of the ball $\mathcal{B}_r(o)$ with the typical cell $V_o$ is given by
\begin{equation}
\mathbb{E}^o[\upsilon_d(\mathcal{B}_r(o)\cap V_o)]=\frac{1}{\lambda}\left(1-\exp(-\lambda\kappa_d r^d)\right),
\label{eq:Mean_Intersection}
\end{equation}
where $\kappa_d=\frac{\pi^{\frac{d}{2}}}{\Gamma(\frac{d}{2}+1)}$ is the volume of the unit-radius ball in $\mathbb{R}^d$.
\end{lemma} 
\begin{proof}
Using \eqref{eq:Moment_Random_Set}, the first moment of the volume of intersection of $\mathcal{B}_r(o)$ with the typical cell $V_o$ can be determined as
\begin{align*}
\mathbb{E}^o[\upsilon_d(\mathcal{B}_r(o)\cap V_o)]&=\int\limits_{\mathbb{R}^d}\mathbb{P}\left(x\in \mathcal{B}_r(o)\cap V_o\right){\rm d}x=\int\limits_{\mathbb{R}^d\cap \mathcal{B}_r(o)}\mathbb{P}\left(x \in V_o\right){\rm d}x\stackrel{(a)}{=}d\kappa_d\int_{0}^r\exp(-\lambda\kappa_d v^d)v^{d-1}{\rm d}v.
\end{align*}
where (a) follows by the change of Cartesian coordinates to polar coordinates and the void probability of the homogeneous PPP. 
\end{proof}
Now, we present the CDF of $\tilde{R}_o$ using the result given in Lemma \ref{lemma:Mean_IntBall_PVCell}.
\begin{thm}
\label{thm:CDF_R_Crofton}
For the homogeneous PPP with intensity $\lambda$ on $\mathbb{R}^d$, the CDF of the distance $\tilde{R}_o$ from the nucleus to a uniformly random point in the 0-cell $\tilde{V}_o$ is 
\begin{equation}
F_{\tilde{R}_o}(r)=1-\exp\left(-\lambda\kappa_d r^d\right),~~~r\geq 0.
\label{eq:CDF_R_Crofton}
\end{equation}
\end{thm}
\begin{proof}
Let $\mathbf{x}_o$ represent the nucleus of $\tilde{V}_o$ and let $\textbf{y}$ represent the uniformly distributed point in $\tilde{V}_o$. 
We note that the distance $\tilde{R}_o=\|\mathbf{x}_o-\mathbf{y}\|$ is less than $r$ when $\mathbf{y}$ lies in the  intersection of the ball $\mathcal{B}_r(\mathbf{x}_o)$ and $\tilde{V}_o$.
Therefore, the CDF of $\tilde{R}_o$ can be written as
\begin{align*}
F_{\tilde{R}_o}(r)&=\mathbb{P}(\tilde{R}_o\leq r)=\mathbb{E}\left[\frac{\upsilon_d(\mathcal{B}_r(\mathbf{x}_o)\cap \tilde{V}_o)}{\upsilon_d(\tilde{V}_o)}\right].
\end{align*}
Now, we define the function $f$ of the PV cell $V_\mathbf{x}$ as the ratio of the volumes of $\mathcal{B}_r(\mathbf{x})\cap V_\textbf{x}$ and $V_\mathbf{x}$. Thus, the function $f$ for the 0-cell and the typical cell, respectively, becomes 
\begin{align*}
f(\tilde{V}_o)=\frac{\upsilon_d(\mathcal{B}_r(\mathbf{x}_o)\cap \tilde{V}_o)}{\upsilon_d(\tilde{V}_o)}~~~\text{and}~~~f(V_o)=\frac{\upsilon_d(\mathcal{B}_r(o)\cap V_o)}{\upsilon_d(V_o)}.
\end{align*}
By substituting the above function in \eqref{eq:Relation_Vo_V}, we obtain
\begin{align*}
\mathbb{E}\left[\frac{\upsilon_d(\mathcal{B}_r(\mathbf{x}_o)\cap \tilde{V}_o)}{\upsilon_d(\tilde{V}_o)}\right]=\lambda\mathbb{E}^{o}[\upsilon_d(\mathcal{B}_r(o)\cap V_o)].
\end{align*}
Finally, we arrive at \eqref{eq:CDF_R_Crofton} by substituting the result of Lemma \ref{lemma:Mean_IntBall_PVCell} in the above equation.  
\end{proof}
Using Theorem \ref{thm:CDF_R_Crofton}, we can calculate the $n$-th moment of the distance $\tilde{R}_o$.
\begin{cor}
\label{cor:Mean_Ro}
For the homogeneous PPP with intensity $\lambda$ on $\mathbb{R}^d$, the $n$-th moment of the distance $\tilde{R}_o$ from the a nucleus to uniformly random point in the 0-cell $\tilde{V}_o$ is 
\begin{align}
\E[\tilde{R}_o^n]=\frac{\Gamma\left(1+\frac{n}{d}\right)}{\left(\lambda\kappa_d\right)^{\frac{n}{d}}}.
\label{eq:Mean_Ro_crofton}
\end{align}
\end{cor}
\begin{remark}
Using the void probability, the distribution of the distance between the origin and the nucleus of $\tilde{V}_o$, {\em say $\mathbf{x}_o$}, can be simply determined as $\mathbb{P}(\|\mathbf{x}_o\|\leq r)=1-\exp(- \lambda\kappa_dr^d)$. However, it does not reveal any information about how the origin is distributed in the 0-cell. While one can intuitively expect the origin to be uniformly distributed in $\tilde{V}_o$, there does not appear to be a straightforward way to prove this. Using \eqref{eq:Relation_Vo_V}, we have presented a simple construction to establish that the distribution of the origin in $\tilde{V}_o$ is in fact that of a uniformly random point in $\tilde{V}_o$.
\end{remark}
In the next section, we present our approach to the exact evaluation of the CDF of $R_o$.  

\section{Distribution of $R_o$} \label{sec:TypicalCell}
We first characterize the CDF of $R_o$ for $d=1$ where the typical cell is completely characterized by the locations of the nearest points on either side of its nucleus. This allows us to explicitly describe the uniformly distributed point in the typical cell $V_o$ and, in turn, determine the CDF of $R_o$. 
In contrast, the structure of the typical cell for $d>1$ is more complex, which makes the distribution of $R_o$ far more difficult to determine, as will be demonstrated in Section \ref{sec:GeneralCase_CDF_R}. 
\subsection{Distribution of $R_o$ for $d=1$}
\label{subsec:1D}
Let $\Phi\triangleq\{x_1,x_2,\dots\}$ be a homogeneous PPP with intensity $\lambda$ on $\mathbb{R}$. Let $x\in\Phi_l\cap\R^-$ and $y\in\Phi_l\cap\R^+$ be left and right neighboring points of the origin (i.e., nucleus of $V_o$), respectively. Since $\Phi$ is a PPP, $|x|$ and $|y|$ are i.i.d.~random variables that follow an exponential distribution with mean $\lambda^{-1}$.
Denote by $R_1=\frac{1}{2}|x|$ and $R_2=\frac{1}{2}|y|$ the distances to the boundary points of the typical cell $V_o$.
Since $|x|$ and $|y|$ are i.i.d.,~ $R_1$ and $R_2$ are also i.i.d.~and follow exponential distribution with parameter $2\lambda$.
Let $\tilde{R}_{1}=\min(R_1,R_2)$ and $\tilde{R}_{2}=\max(R_1,R_2)$. The joint probability density function (pdf) of $\tilde{R}_{1}$ and $\tilde{R}_{2}$ is \cite[Chapter 2]{Mohammad2013}
\begin{align}
f_{\tilde{R}_{1}, \tilde{R}_{2}}(r_1, r_2) = 8\lambda^2 \exp\left(-2\lambda\left(r_1 + r_2\right)\right), \quad 0 \leq r_1\leq r_2.
\label{eq:JointPDF_1D}
\end{align}
The  distribution of the distance $R_o$ from the nucleus to a uniformly random point in the typical cell $V_o$ conditioned on $\tilde{R}_{1}$ and $\tilde{R}_{2}$ is 
\begin{equation}
\mathbb{P}(R_o \leq r\mid \tilde{R}_{1} = r_1,\tilde{R}_{2} = r_2) = \begin{cases} \frac{2r}{r_1 + r_2}, & 0 \leq r \leq r_1 \\
\frac{r + r_1}{r_1 + r_2}, &  r_1 < r \leq r_2 \\
1, &  r_2 < r. \\
\end{cases}
\label{eq:Cond_CDF_1D}
\end{equation}
By deconditioning the above expression with respect to the joint distribution of $\tilde{R}_{1}$ and $\tilde{R}_{2}$, the CDF of $R_o$ is presented in the following theorem.
\begin{figure}
 \centering
\includegraphics[width=.55\textwidth]{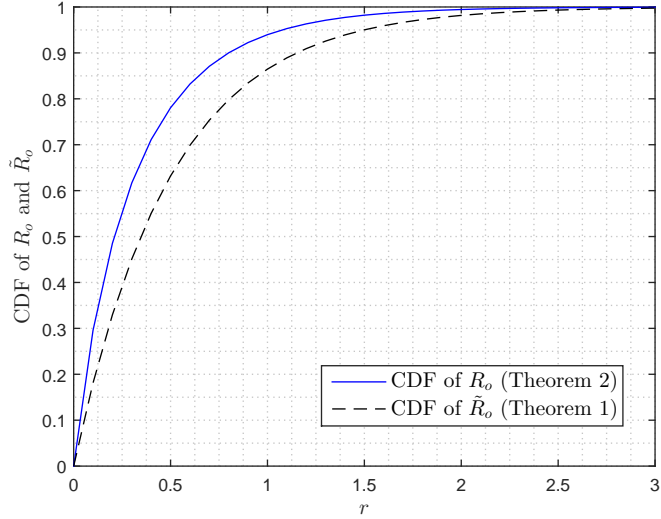}\vspace{-.3cm}
\caption{CDF of $R_o$ and $\tilde{R}_o$ for a unit-intensity Poisson point  process for $d=1$. 
} 
\label{fig:PDF_R_1DPPP}
\end{figure}
\begin{theorem}\label{thm:CDF_R_1D}
For the homogeneous PPP with intensity $\lambda$ on $\mathbb{R}$, the CDF of the distance $R_o$ from the nucleus to a uniformly random point in the typical cell $V_o$ is 
\begin{align}
F_{R_o}(r)&=1-\exp(-2\lambda r)+2\lambda r\exp(-2\lambda r)-4\lambda^2r^2\mathtt{E}_1(2\lambda r),~~ r> 0,\label{eq:CDF_R_1D}
\end{align}
where $\mathtt{E}_1(z)=\int_z^\infty \frac{1}{t}\exp(-t){\rm d}t$ is the exponential integral function.
\end{theorem}
\begin{proof}
Using the expression for the conditional CDF of $R_o$ given in \eqref{eq:Cond_CDF_1D} and the joint pdf of $\tilde{R}_1$ and $\tilde{R}_2$ given in \eqref{eq:JointPDF_1D}, the CDF of $R_o$ can be written as   
\begin{align*}
F_{R_o}(r)=&\int_0^r\int_0^{r_2}8\lambda^2\exp(-2\lambda(r_1+r_2)){\rm d}r_1{\rm d}r_2  +\int_r^\infty\int_0^{r}\frac{r+r_1}{r_1+r_2}8\lambda^2\exp(-2\lambda(r_1+r_2)){\rm d}r_1{\rm d}r_2 \\
&+\int_r^\infty\int_r^{r_2}\frac{2r}{r_1+r_2}8\lambda^2\exp(-2\lambda(r_1+r_2)){\rm d}r_1{\rm d}r_2. \numberthis
\label{eq:CDF_Ro_d1_Integrals}
\end{align*}
Further, using some mathematical simplifications, we obtain the result in  \eqref{eq:CDF_R_1D}. Please refer to Appendix \ref{app:CDF_Ro_d1_Integrals} for more details on the manipulation of the integrals in \eqref{eq:CDF_Ro_d1_Integrals}. 
\end{proof}
In Fig.~\ref{fig:PDF_R_1DPPP}, we provide the plots for the CDFs of $R_o$ and $\tilde{R}_o$. From the figure, it can be seen that the distance $\tilde{R}_o$ stochastically dominates the distance $R_o$. In Section \ref{sec:Approximate_CDF}, we will demonstrate that this difference between the distributions of $\tilde{R}_o$ and $R_o$ diminishes with increasing $d$.
\subsection{Distribution of $R_o$ for $d>1$} 
\label{sec:GeneralCase_CDF_R}
Similar to the distribution of $R_o$ for $d=1$ being derived by conditioning on the nuclei of the neighboring PV cells in Section \ref{subsec:1D}, here we derive the distribution of $R_o$ for $d>1$ by conditioning on the points in a hypersphere centered at the origin such that it includes the nuclei of all neighboring PV cells of $V_o$. We refer to the conditional positions of points in the sphere as the {\em domain configuration}. 
The domain configuration enables the characterization of the shape and size of the PV cell $V_o$ which will be useful in the evaluation of the conditional distribution of $R_o$.   
A similar construction is presented in \cite{Pineda,pineda2008temporal} to study the temporal evolution of the volume of the domain size and free boundary distributions for a PV transformation$^1$ for $d=\{1,2,3\}$\footnote{The simultaneously growing sets of randomly distributed nuclei (realized through PPP) at equal isotropic rate is referred to as the {\em PV transformation}. These sets eventually transform into the PV cells.}. 
In the following subsection, we define the domain configuration and discuss its use for the conditional PV cell characterization.
\subsubsection{Domain Configuration}
First, we define the domain configuration and obtain its probability. Next, we discuss its connection with the conditional shape and size of the PV cell $V_o$.
\begin{definition}
For $\ell>0$, we define the set $\mathcal{C}_\ell^k$ as the set of $k$ points with polar coordinates $(l_i,\mytheta_i)$ such that 
\begin{equation}
\mathcal{C}_\ell^k\equiv\frac{1}{2}\{\Phi\cap \mathcal{B}_{2\ell}(o)\mid \Phi(\mathcal{B}_{2\ell}(o))=k\}.
\label{eq:Domain_Confg_Def}
\end{equation} 
where $l_i$ is the radial coordinate and $\mytheta_i=[\theta_{1i},\dots,\theta_{(d-1)i}]$ are the angular coordinates.
\end{definition}
Thus, the point $\tilde{\textbf{x}}_i\triangleq(l_i,\mytheta_i)\in\mathcal{C}_\ell^k$ bisects the line segment joining $o$ and $\textbf{x}_i \in \Phi\cap \mathcal{B}_{2\ell}(o)$. 
By construction, $l_i\in[0,\ell]$, $ \theta_{(d-1)i}\in[0,2\pi)$ and $\theta_{1i},\dots, \theta_{(d-2)i}\in[0,\pi]$. 
Henceforth, the set  ${\mathcal{C}}_\ell^k$ is referred to as the {\em domain configuration}.  
Since $\Phi$ is a PPP, conditioned  on $\Phi(\mathcal{B}_{2\ell}(o))=k$, the points $\textbf{x}_i\in\Phi\cap \mathcal{B}_{2\ell}(o)$, for $i\in\{1,\dots,k\}$, are distributed uniformly at random independently of each other in $\mathcal{B}_{2\ell}(o)$. Consequently, the $k$ points $\{\tilde{\textbf{x}}_i\}_{i=1}^k$ forming the domain configuration $\mathcal{C}_\ell^k$ are also distributed uniformly at random independently of each other in $\mathcal{B}_\ell(o)$. Using this fact, we can express the pdf of the domain configuration as done next.
 
The differential volume element in $d$ dimensions in polar coordinates is~\cite{blumenson1960derivation}
 $$\Delta=v^{d-1}\sin^{d-2}(\alpha_{1})\dots\sin(\alpha_{d-2}){\rm d}v{\rm d}\alpha_{1}\dots{\rm d}\alpha_{d-1}.$$
Thus, the probability that a point distributed uniformly at random in $\mathcal{B}_\ell(o)$ lies in an infinitesimal region with volume $\Delta_i$ such that $v_i\leq \ell$ is equal to $\frac{\Delta_i}{\kappa_d \ell^d}$. Now, we obtain the pdf of the configuration $\mathcal{C}_\ell^k$ conditioned on $\Phi(\mathcal{B}_l(o))=k$ as
\begin{align}
\mathbb{P}((l_1,\mytheta_1)\in\Delta_1,&\dots,(l_k,\mytheta_k)\in\Delta_k;\ell)\stackrel{(a)}{=}\prod_{i=1}^k\mathbb{P}((l_i,\mytheta_i)\in\Delta_i)\nonumber \\
&\stackrel{(b)}{=}\prod\limits_{i=1}^k \frac{1}{\kappa_d \ell^d}v_i^{d-1}\sin^{d-2}(\alpha_{1i})\dots\sin(\alpha_{(d-2)i}){\rm d}v_i{\rm d}\alpha_{1i}\dots{\rm d}\alpha_{(d-1)i}, ~~\text{for}~0\leq v_i\leq \ell,
\label{eq:Conf_Prob}
\end{align}
where (a) follows from the independence of the elements of $\mathcal{C}_\ell^k$ and (b) follows from the uniform distribution of elements of $\mathcal{C}_\ell^k$ in $\mathcal{B}_\ell(o)$.
\begin{figure}
\centering
\includegraphics[trim={10cm 8cm 8cm 5cm},width=.9\textwidth]{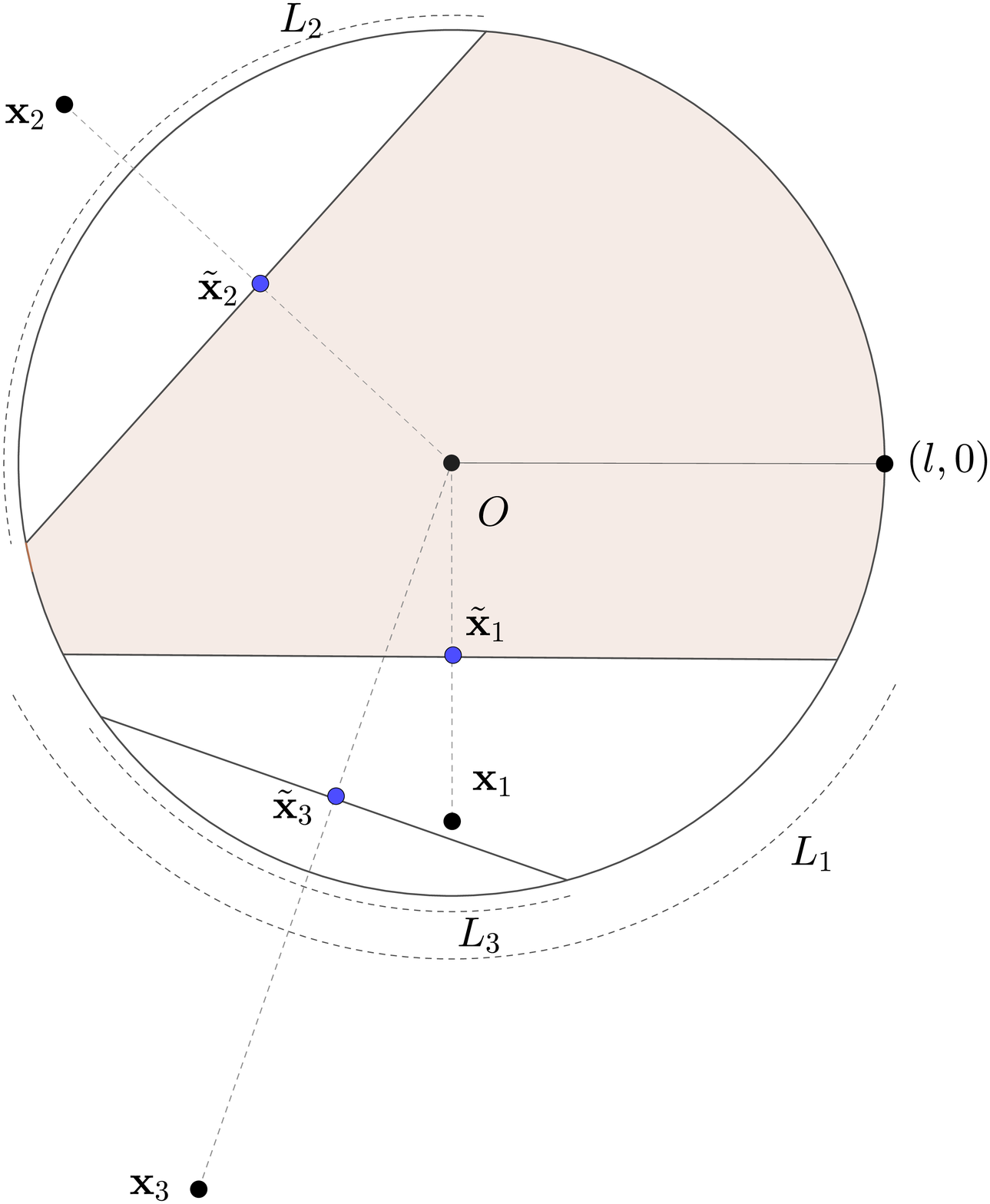}\vspace{-.1cm}
\caption{Illustration of $V_o(\mathcal{C}_\ell^3)\cap \mathcal{B}_\ell(o)$ for $d=2$.}
\label{fig:Illustration}
\end{figure}

\subsubsection{Connections with the Typical Cell}  For an empty domain configuration $\mathcal{C}_\ell^0$, $\mathcal{B}_\ell(o)$ is contained in the typical cell $V_o$. However, a non-empty domain configuration, i.e., $\mathcal{C}_\ell^k$ for $k>0$, contains the mid-points of the chords of $\mathcal{B}_\ell(o)$ formed by the intersection of the edges of typical cell $V_o$ with $\mathcal{B}_\ell(o)$. In addition, the line segments connecting these mid-points to the origin are perpendicular to the corresponding edges. Therefore, the domain configuration provides useful information about the structure of $V_o$. We denote by $V_o(\mathcal{C}_\ell^k)$ the typical cell conditioned on the domain configuration $\mathcal{C}_\ell^k$. 
As $k \rightarrow \infty$, it is easy to see that $V_o(\mathcal{C}_\ell^k)$ becomes deterministic.
However, for any finite $k$, $V_o(\mathcal{C}_\ell^k)$ is in general random because some of its edges may be defined by points of $\Phi$ lying outside $\mathcal{B}_{2\ell}(o)$. That said, conditioning on $\mathcal{C}_\ell^k$ is sufficient to uniquely determine the intersection of $V_o(\mathcal{C}_\ell^k)$ and the ball $\mathcal{B}_\ell(o)$. 
Fig.~\ref{fig:Illustration} illustrates the intersection of the $\mathcal{B}_\ell(o)$ with the cell $V_o(\mathcal{C}_\ell^3)$ for $d=2$.

Let us define $H_{\textbf{x}}$ as the half-space formed by the points in $\R^d$ that are closer to the point $\mathbf{x}\in\Phi$ than the origin, i.e.,
\begin{equation}
H_{\textbf{x}}\triangleq\{\mathbf{y}\in\R^d\mid \|\mathbf{y}-\mathbf{x}\|< \|\mathbf{y}\|\}.
\end{equation}
Now, we denote by $L_i$ the surface (in $d-1$ dimensions) of the spherical cap of $\mathcal{B}_\ell(o)$ such that
\begin{equation}
L_i\triangleq H_{\textbf{x}_i}\cap \partial \mathcal{B}_\ell(o),
\label{eq:Arc_DomainCong}
\end{equation}
where $\partial \mathcal{B}_\ell(o)$ is the boundary of $\mathcal{B}_\ell(o)$.
Note that the surface of the spherical cap is the arc of a circle for $d=2$. From the above definition, it is clear that the point $\tilde{\mathbf{x}}_i\in\mathcal{C}_\ell^k$ is the nearest equidistant point to the origin and $\mathbf{x}_i$ that lies on the supporting hyperplane of $H_{\textbf{x}_i}$. Further, the point $\tilde{\textbf{x}}_i$ is also the center of the $(d-1)$-dimensional chord that forms $L_i$.
This is illustrated in Fig.~\ref{fig:Illustration} for $d=2$.
Now, since $\{\tilde{\textbf{x}}_i\}_{i=1}^k$ are distributed uniformly at random in $\mathcal{B}_\ell(o)$ independently of each other, the corresponding surfaces of the spherical caps $\{L_i\}_{i=1}^k$ have i.i.d.~surface areas\footnote{The surface area in this case is the Lebesgue measure in $d-1$ dimensions.} and are placed uniformly at random on $\partial \mathcal{B}_\ell(o)$. As will be evident in the sequel, this construction will allow us to establish useful conditional geometric properties of the PV cell such as the volume of the intersection of the ball with the PV cell, the conditional distribution of uniformly distributed points within the PV cell, and the shape of large PV cells. We will now use this construction to derive the distribution of $R_o$.
\subsubsection{Distance Distribution} For a given domain configuration $\mathcal{C}_\ell^k$, we define 
\begin{align}
g_k(r;\mathcal{C}_\ell^k) =\upsilon_d(V_o(\mathcal{C}_\ell^k)\cap \mathcal{B}_{r}(o)),
\end{align}
for $0\leq r\leq \ell$, as the volume of the intersection of $\mathcal{B}_{r}(o)$ and cell $V_o(\mathcal{C}_\ell^k)$. As discussed before, $V_o(\mathcal{C}_\ell^k)$ is the typical cell conditioned on the domain configuration $\mathcal{C}_\ell^k$. 
\begin{figure}
\centering
\includegraphics[trim={3cm 1cm 3cm 1cm},width=.9\textwidth]{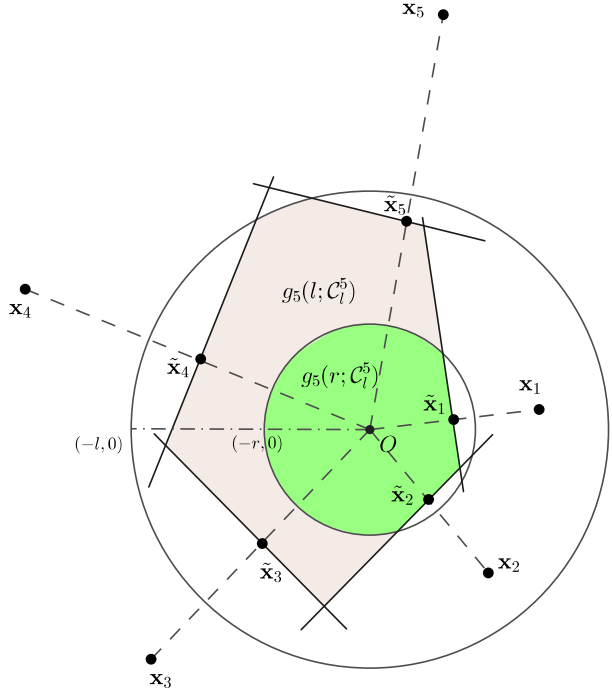}\vspace{-.3cm}
\caption{Illustration of $g_5(r;\mathcal{C}_\ell^5)$ and $g_5(\ell;\mathcal{C}_\ell^5)$ for $d=2$.}
\label{fig:g_k_r}
\end{figure}
\begin{definition} 
Let $R_\ell$ denote the distance from the nucleus of $V_o$ (i.e., the origin) to a uniformly random point in $V_o\cap \mathcal{B}_\ell(o)$. 
\end{definition}
The first main goal is to characterize the CDF of $R_\ell$. Since for $\ell \rightarrow \infty$, $V_o\subset \mathcal{B}_\ell(o)$, the CDF of $R_o$ will  simply be
\begin{align}
F_{R_o} (z) = \lim_{\ell\to\infty} \mathbb{P}(R_\ell \leq z). 
\label{eq:LimitingCase_Dist_of_R}
\end{align}

We first characterize the CDF of $R_\ell$ conditioned on the domain configuration $\mathcal{C}_\ell^k$. This conditional CDF of $R_\ell$ can be expressed as
\begin{align}
F_{R_\ell}(r;\mathcal{C}_\ell^k)&=\frac{\upsilon_d(V_o(\mathcal{C}_\ell^k)\cap \mathcal{B}_{r}(o))}{\upsilon_d(V_o(\mathcal{C}_\ell^k)\cap \mathcal{B}_\ell(o))}=\frac{ g_k(r;\mathcal{C}_\ell^k)}{ g_k(\ell;\mathcal{C}_\ell^k)}, ~~0 \leq r \leq \ell.
\label{eq:CDR_R_Cond_confg}
\end{align} 
Fig. \ref{fig:g_k_r} provides the visual interpretation of $g_k(r,\mathcal{C}_\ell^k)$ and $g_k(l,\mathcal{C}_\ell^k)$ for the typical cell for $d=2$.
The region $g_k(r;\mathcal{C}_\ell^k)$ is shaded in green and the region $g_k(\ell;\mathcal{C}_\ell^k)$ is shaded in brown for $k=5$. Naturally, our next goal is to characterize $g_k(\cdot;\mathcal{C}_\ell^k)$ for which we use $\{L_i\}_{i=1}^k$ given by \eqref{eq:Arc_DomainCong}.

Define the index set $\ncalI(r)$ as the collection of indices $i$ for which $l_i\leq r$. This set points to the collection of the points $\tilde{\textbf{x}}_i$ of the domain configuration that lie inside $\mathcal{B}_r(o)$. It is easy to see that $\cup_{i\in \ncalI(r)} L_i$ represents the portion of $\partial \mathcal{B}_r(o)$ that is outside the typical cell $V_o(\mathcal{C}_\ell^k)$. This can be seen easily from Fig.~\ref{fig:g_k_r} for $d=2$, where the arcs on $\mathcal{B}_r(o)$ corresponding to $\tilde{\textbf{x}}_1\equiv(l_1, \theta_1)$ and $\tilde{\textbf{x}}_2\equiv(l_2, \theta_2)$ do not lie in the cell. Using this insight, we will explicitly characterize the portion of $\partial \mathcal{B}_r(o)$ that lies in $V_o(\mathcal{C}_\ell^k)$, which will then be used to derive the CDF of $R_{\ell}$. This evaluation requires a careful consideration of the overlaps between the surfaces of the spherical caps $\{L_i\}_{i \in \ncalI(r)}$.

Let $\textbf{y}\triangleq(r,\myalpha)$ be the point on the $\partial\mathcal{B}_r(o)$, where $\myalpha = [\alpha_1, \alpha_2, \ldots, \alpha_{d-1}]$.
The Euclidean distance between $\nby\in\partial\mathcal{B}_r(o)$ and $\nbx_i \triangleq (2l_i, \mytheta_i)\in\Phi$ is 
\begin{align*}
d_2(\nby, \nbx_i) = \sqrt{\sum_{n=1}^d\left(y_n - x_{ni}\right)^2},
\end{align*}
where 
\begin{flalign*}
\begin{aligned}
x_{ni} =  
\begin{dcases}
2 l_i \cos(\theta_{1i}); & n= d, \\
2 l_i \prod_{j=1}^{n} \sin(\theta_{ji}) \cos(\theta_{ni}); & 1<n < d, \\
2 l_i \prod_{j=1}^{n-1} \sin(\theta_{ji}); & n = d ,
\end{dcases}
\end{aligned}
~~~~~\text{and}~~~~~
\begin{aligned}
y_{n} =  
\begin{dcases}
r \cos(\alpha_{1}); & n = 1, \\
r \prod_{j=1}^{n} \sin(\alpha_{j}) \cos(\alpha_{n}); & n < d, \\
r \prod_{j=1}^{n-1} \sin(\alpha_{j}); & n = d.\hspace{1cm}
\end{dcases}
\end{aligned}
\end{flalign*}
Let $D_i(l_i,\mytheta_i,\textbf{y})$ be the indicator function taking value 1 if the  point $\mathbf{y}\notin L_i$ or $l_i>r$ (the second condition basically means that $i \notin \ncalI(r)$), otherwise takes value 0. 
Consequently, the points on $\partial \mathcal{B}_r(o)$ that lie in the typical cell $V_o(\mathcal{C}_\ell^k)$ have to be outside of $\{L_i\}_{i=1}^k$. Therefore, we define 
\begin{equation}
D_i\left(l_i,\mytheta_i,\mathbf{y}\right)\triangleq\begin{cases}
\mathbbm{1}\left(d_2(\nby, (2l_i, \mytheta_i)) > r\right);& \text{for~}i \in \ncalI(r) \\
1; &\text{for~}i \notin \ncalI(r). 
\end{cases}
\label{eq:Indicator_Fun}
\end{equation} 
Let $\mathbb{D}=[0,2\pi)\times[0,\pi]^{d-2}$. Using  \eqref{eq:Indicator_Fun}, we can now express the portion of $\partial \mathcal{B}_r(o)$ that belongs to the typical cell $V_o(\mathcal{C}_\ell^k)$ as 
\begin{align}
&\int_{\mathbb{D}}\prod\limits_{i=1}^k D_i\left(l_i,\mytheta_i,\textbf{y}\right)\Delta(\myalpha){\rm d}\myalpha=\frac{1}{r^{d-1}}\upsilon_d(\partial \mathcal{B}_r(o)\cap V_o(\mathcal{C}_\ell^k)),\nonumber
\end{align}
where $\Delta(\myalpha)=\sin^{d-2}(\alpha_{1})\times\dots\times\sin(\alpha_{d-2})$.
Note that $\prod_{i=1}^{k}D_i\left(l_i,\mytheta_i,\mathbf{y}\right)$ is 1 at all points ${\textbf{y}}$, such that $0\leq r\leq z$, lying inside of $\mathcal{B}_z(o)\cap V_o(\mathcal{C}_\ell^k)$, and 0 elsewhere. Thus, the integration of $\prod_{i=1}^{k} D_i\left(l_i,\mytheta_i,\mathbf{y}\right)$ over all the points $\mathbf{y}\in\mathcal{B}_z(o)$ gives the value of $g_k(z;\mathcal{C}_\ell^k)$ for the given domain configuration, i.e.,
\begin{equation}
g_k(z;\mathcal{C}_\ell^k)=\int_{\mathbb{D}}\int_{r=0}^z\prod\limits_{i=1}^k D_i\left(l_i,\mytheta_i,\mathbf{y}\right)r^{d-1}\Delta(\myalpha){\rm d}r{\rm d}\myalpha. 
\label{eq:Conditional_Volume_of_Int_Bz_Vo}
\end{equation}
Using the above results, we present the distance distribution of a uniformly distributed point in $V_o\cap \mathcal{B}_\ell(o)$ conditioned on $\Phi(\mathcal{B}_{2\ell}(o))=k$ in the following lemma. Note that in this lemma, we condition on the number of points that form the domain configuration but not on their locations. 
Let $\mathbf{y}_i=(u_i,\myalpha_i)$ and $\tilde{\mathbb{D}}^d=[0,\ell]\times[0,2\pi)\times[0,\pi]^{d-2}$.
\begin{lemma}
\label{lemma:CONd_CDF}
For given $\ell$, the CDF of $R_\ell$ conditioned on $\Phi(\mathcal{B}_{2\ell}(o))=k$ is 
\begin{align}
F_{R_\ell}(z;k)=\int\limits_{\left(\tilde{\mathbb{D}}^d\right)^k}\frac{ g_k(z;(u_1,\myalpha_1),\dots,(u_k,\myalpha_k))}{g_k(\ell;(u_1,\myalpha_1),\dots,(u_k,\myalpha_k))}\prod\limits_{i=1}^k \frac{1}{\kappa_d \ell^d}u_i^{d-1}\Delta(\myalpha_i){\rm d}\mathbf{y}_i.
 \label{eq:CDF_cond_K}
 \end{align}
 where $g_k(z;(u_1,\myalpha_1),\dots,(u_k,\myalpha_k))$ is given by \eqref{eq:Conditional_Volume_of_Int_Bz_Vo}.
\end{lemma}
\begin{proof}
The CDF of $R_\ell$ conditioned on $\Phi(\mathcal{B}_{2\ell}(o))=k$ is 
$F_{R_\ell}(z;k)=\mathbb{E}_{\mathcal{C}_\ell^k}[F_{R_\ell}(z;\mathcal{C}_\ell^k)]$ where $F_{R_\ell}(z;\mathcal{C}_\ell^k)$ is given by \eqref{eq:CDR_R_Cond_confg}, and the pdf of $\mathcal{C}_\ell^k$  is given in \eqref{eq:Conf_Prob}.
\end{proof}
Using Lemma \ref{lemma:CONd_CDF}, we present the distance distribution of a uniformly distributed point in the typical cell in the following theorem.  
 \begin{thm}
 \label{thm:DDistribution}
For the homogeneous PPP with intensity $\lambda$ on $\mathbb{R}^d$, the CDF of the distance $R_o$ from the nucleus to a uniformly random point in the typical cell $V_o$ is
 \begin{equation}
 F_{R_o}(z)=\lim_{l\to\infty}\sum_{k=0}^\infty F_{R_\ell}(z;k)\mathbb{P}(\Phi(\mathcal{B}_{2\ell}(o))=k),
 \label{eq:CDF_R}
  \end{equation}
  where $F_{R_\ell}(z;k)$ is given in Lemma \ref{lemma:CONd_CDF}. 
\end{thm}
\begin{proof}
The proof follows in two steps. We first take the expectation of the conditional CDF of $R_\ell$, given in Lemma \ref{lemma:CONd_CDF}, over $k$. We then take the limit $\ell \to\infty$ under which this distance distribution of a uniformly distributed point in $V_o\cap \mathcal{B}_\ell(o)$ converges to that of a uniformly distributed point in $V_o$ per  \eqref{eq:LimitingCase_Dist_of_R}. 
\end{proof} 
 \begin{cor}
For the homogeneous PPP with intensity $\lambda$ on $\mathbb{R}^d$, the mean of the distance $R_o$ from  the nucleus to a uniformly random point in the typical cell $V_o$ is
\begin{align}
\mathbb{E}[R_o]=\lim_{\ell\to\infty}\int_{0}^\ell(1-\sum_{k=0}^\infty F_{R_\ell}(z;k)\mathbb{P}(\Phi(\mathcal{B}_{2\ell}(o))=k)){\rm d}z,
\label{eq:Mean_Distance}
\end{align}
where $F_{R_\ell}(z;k)$ is given in Lemma \ref{lemma:CONd_CDF}. 
\end{cor}
\subsubsection{Numerical Results for $d=2$}
\label{sec:Numerical_Results}
\begin{figure}[t]
 \centering
\includegraphics[width=.55\textwidth]{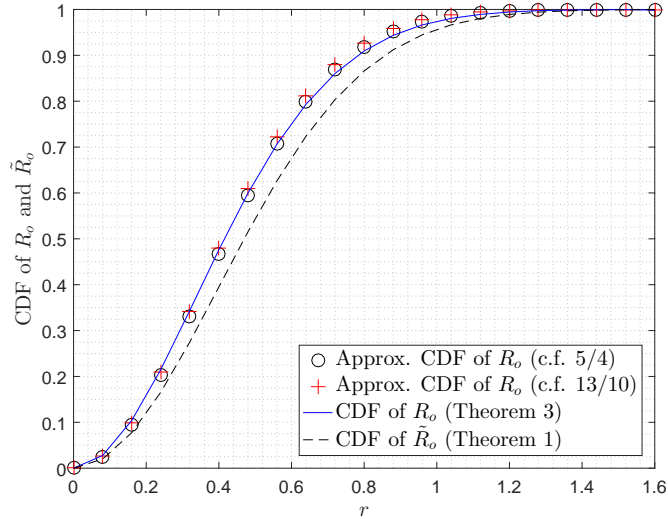}\vspace{-.3cm}
\caption{CDF of $R_o$ and $\tilde{R}_o$ for unit-intensity PPP on $\mathbb{R}^2$.} 
\label{fig:CDF_R}\vspace{.1cm}
\end{figure}
In Fig.~\ref{fig:CDF_R}, we plot the CDF of $\tilde{R}_o$ and the CDF of $R_o$ with $\ell = 1.6$ for $d=2$. This value for $\ell$ is selected because the probability that the distance of the farthest point in the typical cell in $\mathbb{R}^2$ is below $1.6$ is $0.99$ \cite{Calka2002}. The integrals in \eqref{eq:CDF_cond_K} are evaluated numerically using a Monte Carlo integration method. The numerically evaluated mean values of $\tilde{R}_o$ and $R_o$ come out to be $0.500$ and $0.445$. Given the complicated form of the exact CDF of $R_o$, it is desirable to construct closed-form approximations that could be used in obtaining design insights in application-oriented studies. On that note, it has been empirically demonstrated in \cite{Haenggi2017} and \cite{Martin2017_Meta} for $d=2$ that the CDF of $R_o$ can be tightly approximated by $1-\exp(- \pi\rho\lambda r^2)$. It is obtained by introducing a correction factor  (c.f.)~$\rho$ in the CDF of $\tilde{R}_o$ given in \eqref{eq:CDF_R_Crofton}, which reduces to $1-\exp(-\pi\lambda r^2)$ for $d=2$. Furthermore, \cite{Haenggi2017} and \cite{Martin2017_Meta} empirically show that $\rho=13/10$  and $5/4$ provide a close match for the exact CDF of $R_o$. This is also illustrated in Fig.~\ref{fig:CDF_R}. 
Building on these initial insights, we derive the aforementioned c.f.~$\rho$ for the general case of $d$ dimensions in the next section and provide a useful physical interpretation of the resulting value. \vspace{-.3cm}
\section{Approximation of the Distribution of $R_o$ }
\label{sec:Approximate_CDF}
As discussed in the previous section (and shown in Fig. \ref{fig:CDF_R}), the inclusion of an appropriate c.f.~$\rho$ to the CDF of $\tilde{R}_o$ provides a close approximation to the CDF of $R_o$ for $d=2$. Therefore, motivated by this, here we approximate the CDF of $R_o$ with the CDF of $\tilde{R}_o$ by including the c.f.~$\rho_d$ for the $d$-dimensional case. 
That is, the CDF of $R_o$ is approximated as $1-\exp(-\rho_d\lambda \kappa_d r^d)$. We determine the c.f.~$\rho_d$ by matching the $d$-th derivative of the second-order Taylor series expansion of the CDF of $R_o$ at $r=0$. Further, we show that $\rho_d$ is equal to the ratio of the mean volumes of the 0-cell $\tilde{V}_o$ and the typical cell $V_o$. Finally, we show that  $\rho_d\to 1$ as $d\to\infty$. Note that since \cite{Haenggi2017} and \cite{Martin2017_Meta} obtained c.f. $\rho_d$ for $d=2$ through curve fitting, the mathematical treatment provided in this section is new even for the specific case of $d=2$.

For the second-order Taylor series expansion of the CDF of $R_o$, the moments and covariance of the volume of the typical cell $V_o$ and the volume of intersection of $\mathcal{B}_r(o)$ with the typical cell $V_o$ are required. Therefore, before we determine the c.f.~$\rho_d$, we present these intermediate results in the following subsection. 
\subsection{Some Useful Results}
The moments and covariance of the volumes of the typical cell and its intersection with a ball can be derived using \eqref{eq:Moment_Random_Set} along with the void probability of the homogeneous PPP. We first present the second moment of the volume of the typical cell $V_o$ in the following lemma.
\begin{lemma}
\label{lemma:Moments_Volume_PVCell}
The second moment of the volume of the typical cell $V_o$ is
\begin{align}
\mathbb{E}[\upsilon_d(V_o)^2]=4\pi C_{d,2}\int_{0}^{\pi}\int_{0}^\infty\int_{0}^\infty\exp(-\lambda U(v_1,v_2,u))(v_1v_2)^{d-1}(\sin u)^{d-2}{\rm d}v_2{\rm d}v_1{\rm d}u,
\label{eq:2ndMoment_Volume_PVCell}
\end{align} 
where
\begin{align}
U(v_1,v_2,u)=\kappa_dv_1^d+\kappa_dv_2^d-\kappa_dv_1^d\int_0^{\psi_1}\alpha_d\sin^d\psi{\rm d}\psi-\kappa_dv_2^d\int_0^{\psi_2}\alpha_d\sin^d\psi{\rm d}\psi,
\label{eq:Volume_Union_TwoBalls}
\end{align}
$C_{d,2}=\frac{d!}{2(d-2)!}\frac{\kappa_d\kappa_{d-1}}{\kappa_2\kappa_1}$,  $\alpha_d=\frac{\Gamma(\frac{d}{2}+1)}{\Gamma(\frac{1}{2})\Gamma(\frac{d+1}{2})}$, $\psi_1+\psi_2=\pi-u$ and $v_1^d\sin^d\psi_1=v_2^d\sin^d\psi_2$.
\end{lemma}
\noindent Note that $U(v_1,v_2,u)$ represents the union of balls of radii $v_1$ and $v_2$ with centers at angle $u$.
\begin{proof}
Using \eqref{eq:Moment_Random_Set}, we obtain the second moment of $\upsilon_d(V_o)$ as  
\begin{align*}
\mathbb{E}[\upsilon_d(V_o)^2]&=\int_{\mathbb{R}^d}\int_{\mathbb{R}^d}\mathbb{P}(x_1,x_2\in V_o){\rm d}x_1{\rm d}x_2 \\
&=\int_{\mathbb{R}^d}\int_{\mathbb{R}^d}\exp(-\lambda\upsilon_d(\mathcal{B}_{\|x_1\|}(x_1)\cup \mathcal{B}_{\|x_2\|}(x_2))){\rm d}x_2{\rm d}x_1. \numberthis
\end{align*}
Further, following the steps from the proof of \cite[Theorem 3.1]{alishahi2008}, we can obtain \eqref{eq:2ndMoment_Volume_PVCell}. 
\end{proof}
The $n$-th moment of the volume of the intersection of a ball of arbitrary radius with the typical cell  is obtained in \cite[Lemma 4.2]{alishahi2008}. Using this result, we present the first and second moments of $\upsilon_d(\mathcal{B}_r(o)\cap V_o)$ in the following lemma. 
\begin{lemma}
\label{lemma:Moments_Volume_Int_Ball_PVCell}
The first and second moments of the volume of the intersection of the ball $\mathcal{B}_r(o)$ with the typical cell $V_o$ are  
\begin{align}
\mathbb{E}[\upsilon_d(\mathcal{B}_r(o)\cap V_o)]=\frac{1}{\lambda}\left(1-\exp(-\lambda\kappa_d r^d)\right)
 \label{eq:1stMoment_Volume_Int_Ball_PVCell}
\end{align}
and 
\begin{align}
\mathbb{E}[\upsilon_d(\mathcal{B}_r(o)\cap V_o)^2]=4\pi C_{d,2}\int_{0}^{\pi}\int_{0}^r\int_{0}^r\exp(-\lambda U(v_1,v_2,u))(v_1 v_2)^{d-1}(\sin u)^{d-2}{\rm d} v_2 {\rm d} v_1 {\rm d}u,
 \label{eq:2ndMoment_Volume_Int_Ball_PVCell}
\end{align}
where $U(v_1,v_2,u)$ is given by \eqref{eq:Volume_Union_TwoBalls}.
\end{lemma}
\begin{proof}
The first moment in \eqref{eq:1stMoment_Volume_Int_Ball_PVCell} follows from Lemma \ref{lemma:Mean_IntBall_PVCell}. Similar to the  second moment of the volume of the typical cell derived in Lemma \ref{lemma:Moments_Volume_PVCell}, the second moment of the volume of $\mathcal{B}_r(o)\cap V_o$ can be determined as
 \begin{align*}
 \mathbb{E}[\upsilon_d(\mathcal{B}_r(o)\cap V_o)^2]&=\int_{\mathbb{R}^d}\int_{\mathbb{R}^d}\mathbb{P}(x_1,x_2\in \mathcal{B}_r(o)\cap V_o){\rm d}x_1{\rm d}x_2\nonumber\\
 &=\int_{\mathbb{R}^d\cap \mathcal{B}_r(o)}\int_{\mathbb{R}^d\cap \mathcal{B}_r(o)}\exp\left(-\lambda\upsilon_d(\mathcal{B}_{\|x_1\|}(x_1)\cup \mathcal{B}_{\|x_2\|}(x_2))\right){\rm d}x_1{\rm d}x_2.
 \end{align*}
Following similar steps as in Lemma \ref{lemma:Moments_Volume_PVCell} completes the proof.
\end{proof}
In \cite[Lemma~3.1]{Olsbo2007}, the correlation between the volume of the \textit{typical Stienen sphere} and the volume of the typical cell is derived. Using the approach of \cite{Olsbo2007}, we provide the covariance of the volumes of $\mathcal{B}_r(o)\cap V_o$ and $V_o$ in the following lemma.
\begin{lemma}
\label{lemma:Cov_Int_Ball_PVCell_1}
The covariance of the volume of the intersection of $\mathcal{B}_r(o)$ with the typical cell $V_o$ and the volume of the typical cell $V_o$ is 
\begin{align}
\mathtt{Cov}[\upsilon_d(\mathcal{B}_r(o)\cap V_o),\upsilon_d(V_o)]=&  \frac{1}{2}\mathtt{Var}[\upsilon_d(V_o)] - \frac{1}{2\lambda^2}\left(1-2\exp(-\lambda\kappa_d r^d)\right) \label{eq:Cov_Int_Ball_PVCell_1}\\
& + 2\pi C_{d,2}\int_{0}^{\pi}\int_{0}^r\int_{0}^r\exp(-\lambda U(v_1,v_2,u))(v_1v_2)^{d-1}(\sin u)^{d-2}{\rm d}v_2{\rm d}v_1{\rm d}u \nonumber\\
& - 2\pi C_{d,2}\int_{0}^{\pi}\int_{r}^\infty\int_{r}^\infty\exp(-\lambda U(v_1,v_2,u))(v_1v_2)^{d-1}(\sin u)^{d-2}{\rm d}v_2{\rm d}v_1{\rm d}u,\nonumber
\end{align}
where $U(v_1,v_2,u)$ is given by \eqref{eq:Volume_Union_TwoBalls}.
\end{lemma}
\begin{proof}
Let $\hat{V}_o(r)=V_o\setminus V_o\cap \mathcal{B}_r(o)$. The variance of the volume of $\hat{V}_o(r)$ is
\begin{align*}
\mathtt{Var}[\upsilon_d(\hat{V}_o(r))]&=\mathtt{Var}[\upsilon_d({V_o})-\upsilon_d(\mathcal{B}_r(o)\cap V_o)]\nonumber\\
&=\mathtt{Var}[\upsilon_d(V_o)]+\mathtt{Var}[\upsilon_d(\mathcal{B}_r(o)\cap V_o)]-2\mathtt{Cov}[\upsilon_d(\mathcal{B}_r(o)\cap V_o),\upsilon_d(V_o)].
\end{align*}
This implies
\begin{align}
\mathtt{Cov}[\upsilon_d(\mathcal{B}_r(o)\cap V_o),\upsilon_d(V_o)]=\frac{1}{2}\mathtt{Var}[\upsilon_d(V_o)]+\frac{1}{2}\mathtt{Var}[\upsilon_d(\mathcal{B}_r(o)\cap V_o)]-\frac{1}{2}\mathtt{Var}[\upsilon_d(\hat{V}_o(r))].
\label{eq:Cov_Int_Ball_PVCell}
\end{align}
Using Lemma \ref{lemma:Moments_Volume_Int_Ball_PVCell},  the variance of $\upsilon_d(\mathcal{B}_r(o)\cap V_o)$ can be expressed as
\begin{align}
\mathtt{Var}[\upsilon_d(\mathcal{B}_r(o)\cap V_o)]&=4\pi C_{d,2}\int_{0}^{\pi}\int_{0}^r\int_{0}^r\exp(-\lambda U(v_1,v_2,u))(v_1v_2)^{d-1}(\sin u)^{d-2}{\rm d}v_2{\rm d}v_1{\rm d}u\nonumber\\
&~~~~-\frac{1}{\lambda^2}\left(1-\exp(-\lambda\kappa_d r^d)\right)^2.
\label{eq:Var_Volume_Int_Ball_PVCell}
\end{align} 
Now, we obtain the mean and variance of $\hat{V}_o(r)$. Using \eqref{eq:1stMoment_Volume_Int_Ball_PVCell}, the first moment becomes
\begin{equation}
\mathbb{E}[\upsilon_d(\hat{V}_o(r))]=\mathbb{E}[\upsilon_d(V_o)-\upsilon_d(\mathcal{B}_r(o)\cap V_o)]=\frac{1}{\lambda}\exp(-\lambda\kappa_d r^d).
\label{eq:1stMoment_Volume_Rem_Int_PVCell}
\end{equation}

Using \eqref{eq:Moment_Random_Set}, we can obtain the second moment as 
\begin{align}
\mathbb{E}[\upsilon_d(\hat{V}_o(r))^2]&=\int_{\mathbb{R}^d}\int_{\mathbb{R}^d}\mathbb{P}(r< \|x_1\|,r< \|x_2\|,x_1,x_2\in V_o(r)){\rm d}x_2{\rm d}x_1\nonumber\\
&=\int_{\mathbb{R}^d\setminus \mathcal{B}_r(o)}\int_{\mathbb{R}^d\setminus \mathcal{B}_r(o)}\mathbb{P}(x_1,x_2\in V_o){\rm d}x_2{\rm d}x_1\nonumber\\
&\stackrel{(a)}{=}4\pi C_{d,2}\int_{0}^{\pi}\int_{r}^\infty\int_{r}^\infty\exp(-\lambda U(v_1,v_2,u))(v_1v_2)^{d-1}(\sin u)^{d-2}{\rm d}v_2{\rm d}v_1{\rm d}u,
\label{eq:2ndMoment_Volume_Rem_Int_PVCell}
\end{align}
where (a) follows from the same steps as in Lemma \ref{lemma:Moments_Volume_PVCell} for the second moment of the volume of typical cell $V_o$. 
Lastly, substituting \eqref{eq:Var_Volume_Int_Ball_PVCell}, \eqref{eq:1stMoment_Volume_Rem_Int_PVCell} and \eqref{eq:2ndMoment_Volume_Rem_Int_PVCell} in \eqref{eq:Cov_Int_Ball_PVCell} completes the proof.
\end{proof}
Since we use $1-\exp(-\rho_d\lambda\kappa_d r^d)$ for the approximation of CDF of $R_o$, the c.f.~$\rho_d$ is determined by matching the $d$-th derivative of the second-order approximation of the CDF of $R_o$ at $r=0$. As the second-order Taylor series expansion of the CDF includes the covariance term given in Lemma \ref{lemma:Cov_Int_Ball_PVCell_1}, we first provide its $d$-th derivative at $r=0$ in the following lemma.
\begin{lemma}
\label{lemma:Cov_derivative}
The $d$-th derivative of the covariance of the volume of the  intersection of $\mathcal{B}_r(o)$ with the typical cell $V_o$ and the volume of typical cell $V_o$ w.r.t. $r$ is zero at $r=0$.
\end{lemma}
\begin{proof}
Using Lemma \ref{lemma:Cov_Int_Ball_PVCell_1}, we can write
\begin{align}
\frac{{\rm d}^d}{{\rm d}r^d}\mathtt{Cov}[\upsilon_d(\mathcal{B}_r(o)\cap V_o),\upsilon_d(V_o)]\bigg|_{r=0}=\frac{{\rm d}^d}{{\rm d}r^d}\frac{1}{\lambda^2}\exp(-\lambda\kappa_d r^d)\bigg|_{r=0} +\frac{{\rm d}^d}{{\rm d}r^d}  \left( f_1(r) -  f_2(r)\right)\bigg|_{r=0},
\label{eq:derivative_of_Cov_1}
\end{align}
 where
\begin{align*}
f_1(r)= \int_{0}^r\int_{0}^rg(v_1,v_2){\rm d}v_2{\rm d}v_1,\text{~~and~~} f_2(r)=\int_{r}^\infty\int_{r}^\infty g(v_1,v_2){\rm d}v_2{\rm d}v_1,
\end{align*}
 such that 
 \begin{align*}
 g(v_1,v_2)=2\pi C_{d,2}\int_{0}^{\pi}\exp(-\lambda U(v_1,v_2,u))(v_1v_2)^{d-1}(\sin u)^{d-2}{\rm d}u.
\end{align*}  
Further,
\begin{equation}
\frac{{\rm d}^d}{{\rm d}r^d}\frac{1}{\lambda^2}\exp(-\lambda\kappa_d r^d)\bigg|_{r=0}=-\frac{1}{\lambda}d!\kappa_d=-\frac{1}{\lambda}2\pi^{\frac{d}{2}}\frac{\Gamma(d)}{\Gamma(\frac{d}{2})}.
\label{eq:Derivetive_Cov_1stTerm}
\end{equation}
Now, differentiating $f_1$ w.r.t. $r$, we obtain
\begin{align*}
\frac{{\rm d}}{{\rm d}r}f_1(r)&=\frac{{\rm d}}{{\rm d}r} \int_{0}^r{\rm d}v_1\int_{0}^{r}g(v_1,v_2){\rm d}v_2\\
&\stackrel{(a)}{=}\int_{0}^{r}g(r,v_2){\rm d}v_2 + \int_{0}^r{\rm d}v_1\frac{{\rm d}}{{\rm d}r}\int_{0}^{r}g(v_1,v_2){\rm d}v_2\\
&\stackrel{(b)}{=} \int_{0}^{r}g(r,v_2){\rm d}v_2 +  \int_{0}^rg(v_1,r){\rm d}v_1,
\end{align*}
where (a) and (b) are obtained using the successive application of Leibniz's integral rule. Again differentiating, we obtain
\begin{align*}
\frac{{\rm d^2}}{{\rm d}r^2}f_1(r)&= \frac{{\rm d}}{{\rm d}r} \int_{0}^{r}g(r,v_2){\rm d}v_2 +  \frac{{\rm d}}{{\rm d}r}\int_{0}^rg(v_1,r){\rm d}v_1\\
&\stackrel{(a)}{=} 2g(r,r) +  \int_0^r \frac{{\rm d}}{{\rm d}r} g(r,v_2){\rm d}v_2 +   \int_0^r \frac{{\rm d}}{{\rm d}r}g(v_1,r){\rm d}v_1, 
\end{align*}
where (a) is obtained using Leibniz's integral rule.
Similarly, we get
\begin{align*}
\frac{{\rm d}^3}{{\rm d}r^3}f_1(r)&= 4\frac{{\rm d}}{{\rm d}r}g(r,r) +  \int_0^r \frac{{\rm d}^2}{{\rm d}r^2} g(r,v_2){\rm d}v_2 +   \int_0^r \frac{{\rm d}^2}{{\rm d}r^2}g(v_1,r){\rm d}v_1, \\
\frac{{\rm d}^4}{{\rm d}r^4}f_1(r)&= 6\frac{{\rm d}^2}{{\rm d}r^2}g(r,r) +  \int_0^r \frac{{\rm d}^3}{{\rm d}r^3} g(r,v_2){\rm d}v_2 +   \int_0^r \frac{{\rm d}^3}{{\rm d}r^3}g(v_1,r){\rm d}v_1.
\end{align*}
Thus, in general, we have
\begin{align*}
\frac{{\rm d}^d}{{\rm d}r^d}f_1(r)=2(d-1)\frac{{\rm d}^{(n-2)}}{{\rm d}r^{(n-2)}}g(r,r)+ \int_0^r \frac{{\rm d}^{(n-1)}}{{\rm d}r^{(n-1)}} g(r,v_2){\rm d}v_2 +   \int_0^r \frac{{\rm d}^{(n-1)}}{{\rm d}r^{(n-1)}}g(v_1,r){\rm d}v_1. 
\end{align*}
Following similar steps, we obtain the $d$-fold derivative of $f_2$ w.r.t. $r$ as 
\begin{align*}
\frac{{\rm d}^d}{{\rm d}r^d}f_2(r)=2(d-1)\frac{{\rm d}^{(n-2)}}{{\rm d}r^{(n-2)}}g(r,r)- \int_r^\infty \frac{{\rm d}^{(n-1)}}{{\rm d}r^{(n-1)}} g(r,v_2){\rm d}v_2 -\int_r^\infty \frac{{\rm d}^{(n-1)}}{{\rm d}r^{(n-1)}}g(v_1,r){\rm d}v_1. 
\end{align*}
Subtracting $\frac{{\rm d}^d}{{\rm d}r^d}f_2(r)$ from $\frac{{\rm d}^d}{{\rm d}r^d}f_1(r)$, we get 
\begin{align}
\frac{{\rm d}^d}{{\rm d}r^d}  \left( f_1(r) -  f_2(r)\right)=\int_0^\infty \frac{{\rm d}^{(d-1)}}{{\rm d}r^{(d-1)}} g(r,v_2){\rm d}v_2 + \int_0^\infty \frac{{\rm d}^{(d-1)}}{{\rm d}r^{(d-1)}}g(v_2,r){\rm d}v_1.
\label{eq:derivative_f1f2}
\end{align}
Now, we obtain the $(d-1)$-th derivative of $g(r,v_2)$ at $r=0$ as  
\begin{align*}
\frac{{\rm d}^{(d-1)}}{{\rm d}r^{(d-1)}} g(r,v_2)\bigg|_{r=0}&=2\pi C_{d,2}\frac{{\rm d}^{(d-1)}}{{\rm d}r^{(d-1)}}\int_{0}^{\pi}\exp(-\lambda U(r,v_2,u))(r_1v_2)^{d-1}(\sin u)^{d-2}{\rm d}u\bigg|_{r=0}\\
&=2\pi C_{d,2}\int_{0}^{\pi}\frac{{\rm d}^{(d-1)}}{{\rm d}r^{(d-1)}}\exp(-\lambda U(r,v_2,u))(rv_2)^{d-1}\bigg|_{r=0}(\sin u)^{d-2}{\rm d}u\\ 
&=2\pi(d-1)! C_{d,2}\int_{0}^{\pi}\exp(-\lambda U(0,v_2,u))v_2^{d-1}(\sin u)^{d-2}{\rm d}u \\
&\stackrel{(a)}{=}2\pi(d-1)!C_{d,2}v_2^{d-1}\exp(-\lambda\kappa_d v_2^d)\int_{0}^{\pi}(\sin u)^{d-2}{\rm d}u\\
&\stackrel{(b)}{=}d\pi^d\frac{\Gamma(d)}{\Gamma(\frac{d}{2})}\frac{1}{\Gamma\left(\frac{d}{2}+1\right)}v_2^{d-1}\exp(-\lambda\kappa_d v_2^d),
\end{align*}
where (a) follows due to $U(0,v_2,u)=\kappa_dv_2^d$ and (b) follows using $\int_0^\pi(\sin u)^{d-2}{\rm d}u=\sqrt{\pi}\frac{\Gamma(\frac{d-1}{2})}{\Gamma(\frac{d}{2})}$ \cite[Eq. 3.62.5]{gradshteyn2014table}.
Now, using the above expression along with $g(r,x)=g(x,r)$ and 
\begin{align*}
\int_0^\infty v^{d-1}\exp(-\lambda\kappa_d v^d){\rm d}v=\frac{1}{d\lambda\kappa_d}\int_0^\infty \exp(-t){\rm d}t=\frac{1}{d\lambda\kappa_d}=\frac{\Gamma(\frac{d}{2}+1)}{d\lambda\pi^{\frac{d}{2}}},
\end{align*} 
we can write \eqref{eq:derivative_f1f2} at $r=0$ as 
\begin{align}
\frac{{\rm d}^d}{{\rm d}r^d}  \left( f_1(r) -  f_2(r)\right)\bigg|_{r=0}&=\frac{1}{\lambda}2\pi^{\frac{d}{2}}\frac{\Gamma(d)}{\Gamma(\frac{d}{2})}.
\label{eq:Derivetive_Cov_2ndTerm}
\end{align}
Finally, the substitution of \eqref{eq:Derivetive_Cov_1stTerm} and \eqref{eq:Derivetive_Cov_2ndTerm} in \eqref{eq:derivative_of_Cov_1} completes the proof.
\end{proof}
\subsection{Approximate CDF of $R_o$}
\label{subsec:ApproximateCDF}
Now, in the following theorem we determine the c.f.~of the approximated CDF of $R_o$, which is the main result of this section.
\begin{thm}
\label{thm:App_CDF_R_Typical}
For the homogeneous PPP with intensity $\lambda$ on $\mathbb{R}^d$, the approximate CDF of the distance $R_o$ from the nucleus to a uniformly random point in the typical cell $V_o$ is 
\begin{align}
F_{R_o}(r)\approx 1-\exp(-\rho_d\lambda\kappa_d r^d),
\label{eq:CDF_R_apprx}
\end{align}
 where $\rho_d$ is the c.f. obtained by matching the $d-th$ derivative of \eqref{eq:CDF_R_apprx} with that of the second-order Taylor series expansion of the exact CDF of $R_o$ at $r=0$ and is given by
\begin{align}
\rho_d= 1+\frac{\mathtt{Var}[\upsilon_d(V_o)]}{\mathbb{E}[\upsilon_d(V_o)]^2}.
\label{eq:CorrectionFactor}
\end{align}
\end{thm}
\begin{proof}
The second order Taylor series expansion of the bivariate function $f(Z_1,Z_2)=\frac{Z_1}{Z_2}$ around the mean ($\bar{z}_1,\bar{z}_2$) can be written as
\begin{align*}
f(Z_1,Z_2)\approx\frac{\bar{z}_1}{\bar{z}_2} + \frac{1}{\bar{z}_2}(Z_1-\bar{z}_1) -\frac{\bar{z}_1}{\bar{z}_2^2}(Z_2-\bar{z}_2)+\frac{1}{\bar{z}_2^2}(Z_1-\bar{z}_1)(Z_2-\bar{z}_2) + \frac{\bar{z}_1}{\bar{z}_2^3}(Z_2-\bar{z}_2)^2.
\end{align*}
Taking expectation of $f(Z_1,Z_2)$ w.r.t. $Z_1$ and $Z_2$, we get
\begin{equation}
\mathbb{E}[f(Z_1,Z_2)]\approx\frac{\bar{z}_1}{\bar{z}_2} - \frac{1}{\bar{z}_2^2}\mathtt{Cov}[z_1,z_2] + \frac{\bar{z}_1}{\bar{z}_2^3}\mathtt{Var}[z_2].
\label{eq:TaylorExpansion}
\end{equation} 
The CDF of $R_o$ is 
$$F_{R_o}(r)=\mathbb{E}\left[\frac{\upsilon_d(\mathcal{B}_r(o)\cap V_o)}{\upsilon_d(V_o)}\right].$$
 Therefore, using \eqref{eq:TaylorExpansion},  the second-order Taylor series expansion of $F_{R_o}(r)$  around the mean ($\E[\upsilon_d(\mathcal{B}_r(o)\cap V_o)],\E[\upsilon_d(V_o)]$) can be written as 
\begin{align*}
F_{R_o}(r)\approx\frac{\mathbb{E}[\upsilon_d(\mathcal{B}_r(o)\cap V_o)]}{\mathbb{E}[\upsilon_d(V_o)]}\left[1+\frac{\mathtt{Var}[\upsilon_d(V_o)]}{\mathbb{E}[\upsilon_d(V_o)]^2}\right]-\frac{\mathtt{Cov}[\upsilon_d(\mathcal{B}_r(o)\cap V_o),\upsilon_d(V_o)]}{\mathbb{E}[\upsilon_d(V_o)]^2}.
\end{align*}
Using Lemma \ref{lemma:Moments_Volume_PVCell} and Lemma \ref{lemma:Moments_Volume_Int_Ball_PVCell}, we obtain
\begin{align}
F_{R_o}(r) \approx & \left(1-\exp(-\lambda\kappa_d r^d)\right)\left[1+\frac{\mathtt{Var}[\upsilon_d(V_o)]}{\mathbb{E}[\upsilon_d(V_o)]^2}\right]-\frac{\mathtt{Cov}[\upsilon_d(\mathcal{B}_r(o)\cap V_o),\upsilon_d(V_o)]}{\mathbb{E}[\upsilon_d(V_o)]^2}.\label{eq:Cov_Int_Ball_PVCell_1}
\end{align}
Now, as $1-\exp(-\rho_d\lambda\kappa_d r^d)$ is considered for the approximation, we determine the c.f.~$\rho_d$ by matching the $d$-th derivatives of $1-\exp(-\rho_d\lambda\kappa_d r^d)$ and $F_{R_o}(r)$ at $r=0$ as 
\begin{align*}
\rho_d=\frac{1}{d!\lambda\kappa_d}\frac{{\rm d}^d}{{\rm d}r^d}F_{R_o}(r)\bigg|_{r=0}.
\end{align*}
Therefore, using  \eqref{eq:Cov_Int_Ball_PVCell_1} and Lemma \ref{lemma:Cov_derivative} we have
\begin{align*}
\rho_d= 1+\frac{\mathtt{Var}[\upsilon_d(V_o)]}{\mathbb{E}[\upsilon_d(V_o)]^2}.
\end{align*}
This completes the proof.
\end{proof}
Before giving the numerical validation of the approximated CDF of $R_o$, we present  the approximated $n$-th moment of the distance $R_o$ and some useful observations about the c.f.~in the following corollaries.
\begin{cor}
\label{cor:Mean_R}
For the homogeneous PPP with intensity $\lambda$ on $\mathbb{R}^d$, the $n$-th moment of the distance $R_o$ from the nucleus to a uniformly random point in the typical cell $V_o$ is approximately 
\begin{align}
\E[R_o^n]\approx\frac{\Gamma\left(1+\frac{n}{d}\right)}{\left(\rho_d\lambda\kappa_d\right)^{\frac{n}{d}}}.
\label{eq:Mean_Ro_Typical}
\end{align}
\end{cor}
\begin{cor}
\label{cor:CorrectionFactor_ratio}
For the homogeneous PPP with intensity $\lambda$ on $\mathbb{R}^d$, the CDF of the distance $R_o$ from the nucleus to a uniformly random point in the typical cell $V_o$ can be approximated as $1-\exp(-\lambda\kappa_d\rho_dr^d)$ where the c.f.~$\rho_d$ is equal to the ratio of the mean volumes of the 0-cell and the typical cell, i.e.,
\begin{align}
\rho_d= \frac{\mathbb{E}[\upsilon_d(\tilde{V}_o)]}{\mathbb{E}[\upsilon_d(V_o)]}.
\label{eq:CorrectionFactor_RatioVolumes}
\end{align}
\end{cor}
\begin{proof}
 From \cite[Equation 2.5]{MECKE1999}, we have 
\begin{align*}
\mathbb{E}[\upsilon_d(\tilde{V}_o)]=\mathbb{E}[\upsilon_d(V_o)] + \frac{\mathtt{Var}[\upsilon_d(V_o)]}{\mathbb{E}[\upsilon_d(V_o)]}.
\end{align*} 
Substituting the above expression in \eqref{eq:CorrectionFactor} gives \eqref{eq:CorrectionFactor_RatioVolumes}.
\end{proof}
\begin{cor}
\label{cor:CDF_R_d_Infinity}
The c.f.~$\rho_d$ approaches one as $d$ approaches infinity, i.e., $\lim\limits_{d\to\infty}\rho_d=1$.
\end{cor}
\begin{proof}
Using \cite[Theorem 3.1]{alishahi2008}, we can write
\begin{align*}
\lim_{d\to\infty}\mathtt{Var}[\upsilon_d(V_o)]=0.
\end{align*}
Since, the mean volume of the PV cell is $\lambda^{-1}$ for any $d$, the proof directly follows using \eqref{eq:CorrectionFactor} and above result.
\end{proof}
\begin{figure}
 \centering
\includegraphics[width=.55\textwidth]{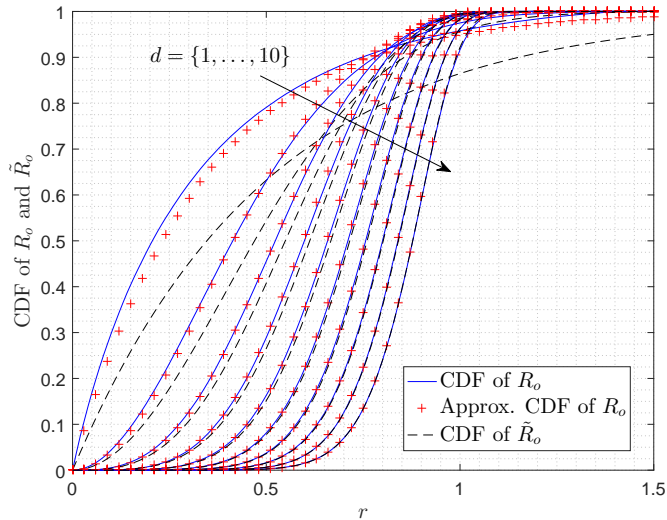}\vspace{-.3cm}
\caption{CDF of $R_o$ and $\tilde{R}_o$ for unit-intensity PPP on $\mathbb{R}^d$ where $d\in\{1,\dots,10\}$. The CDF of $\tilde{R}_o$ and approximate CDF of $R_o$ are given in Theorem \ref{thm:CDF_R_Crofton} and Theorem \ref{thm:App_CDF_R_Typical}, respectively.}
\label{fig:Approx_CDF_R}
\end{figure}

\begin{remark}
From  \eqref{eq:Mean_Ro_crofton} and  \eqref{eq:Mean_Ro_Typical}, it is clear that the ratio of the means of $\tilde{R}_o$ and $R_o$ is approximately $\sqrt[d]{\rho_d}$. Therefore, using Corollary \ref{cor:CorrectionFactor_ratio}, we can infer that the ratio of the means of $\tilde{R}_o$ and $R_o$ is approximately equal to the $d$-th root of the ratio of the mean volumes of the 0-cell $\tilde{V}_o$ and the typical cell $V_o$. In other words, the distance from the nucleus to a uniformly random point in the typical cell scales with the distance from the nucleus to a uniformly random point in the 0-cell by a factor equal to the $d$-th root of the ratio of the mean volumes of the 0-cell $\tilde{V}_o$ and the typical cell $V_o$.

\end{remark}

\subsection{Numerical Comparisons}
\label{sec:Numerical_Results_Appx}
For the numerical evaluation of the approximated CDF of $R_o$, we obtain the c.f.~$\rho_d$ using \eqref{eq:CorrectionFactor} for which the mean and variance of the volume of the typical cell are evaluated using Lemma \ref{lemma:Moments_Volume_PVCell}. Fig. \ref{fig:Approx_CDF_R}  validates the accuracy of the approximated CDF of $R_o$ by comparing it with the Monte Carlo simulations for the cases of $d\in\{1,\dots,10\}$.  
Fig. \ref{fig:Approx_CDF_R} clearly indicates that the CDF of $R_o$ gradually approaches that of $\tilde{R}_o$ as $d$ increases. 
Further, Table \ref{table:Accuracy_CDF_Ro} verifies the accuracy of the approximated mean and variance of $R_o$ (obtained using Corollary \ref{cor:Mean_R}) for $d\in\{1,\dots,10\}$.
For $d=2$, the obtained mean value of $R_o$ is $0.442$ which is also close to the mean values $0.438$ and $0.447$ obtained using the curve-fitted  c.f.s~ $13/10$ and $5/4$ of \cite{Haenggi2017} and  \cite{Martin2017_Meta}, respectively.
\begin{table}
\centering
\caption{Accuracy of Approximated Mean and Variance of $R_o$.}\vspace{.1cm}
\label{table:Accuracy_CDF_Ro}
\scriptsize{
\begin{tabular}{|l|l|l|l|l|l|l|l|l|l|l|l|} 
\hline
\multicolumn{2}{|l|}{$d$}     		& 1     & 2     & 3     & 4     & 5     & 6     & 7     & 8     & 9     & 10       \\ 
\hline
\multicolumn{2}{|l|}{$\rho_d$} 		& 1.500  & 1.285  & 1.171  & 1.128  & 1.079 & 1.062  & 1.043  & 1.032 & 1.029 & 1.018  \\ 
\hline
\multirow{2}{*}{$\mathbb{E}[R_o]$} & Exact  & 0.305 & 0.445 & 0.529 & 0.595 & 0.651 & 0.701 & 0.749 & 0.798 & 0.831 & 0.873  \\ 
\cline{2-12}
                                   & Approx. & 0.333 & 0.442 & 0.524 & 0.591 & 0.648 & 0.698 & 0.745 & 0.789 & 0.829 & 0.862  \\ 
\hline
\multirow{2}{*}{$\mathtt{Var}[R_o]$} & Exact  & 0.090 & 0.058 & 0.038 & 0.028 & 0.022 & 0.019 & 0.016 & 0.014 & 0.013 & 0.012  \\ 
\cline{2-12}
                                   & Approx. & 0.111 & 0.053 & 0.036 & 0.028 & 0.022 & 0.018 & 0.015 & 0.013 & 0.012 & 0.011  \\
\hline
\end{tabular}}
\end{table}

\section{Limiting Shape of Large PV Cells}
\label{sec:LimitingShape}
Thus far, we have presented an exact characterization of the CDFs of $\tilde{R}_o$ and $R_o$ in Sections \ref{sec:CroftonCell} and \ref{sec:TypicalCell} and a closed-form approximation for the multi-integral exact expression for the CDF of $R_o$ in Section \ref{sec:Approximate_CDF}. 
It is worth noting that the conditioning on the $k$ points of $\Phi$ in the $\mathcal{B}_{2\ell}(o)$, defined as the domain configuration $\mathcal{C}_\ell^k$ (see \eqref{eq:Domain_Confg_Def}), allowed us to construct the set of surfaces of the spherical caps $\{L_i\}_{i=1}^{k}$ on the ball $\mathcal{B}_\ell(o)$ as in \eqref{eq:Arc_DomainCong}. This helps in determining the conditional volume of the typical cell $V_o$ and thus the conditional CDF of $R_o$. 
It is easy to observe that some points of the domain configuration $\mathcal{C}_\ell^k$ are the closest points on some boundaries of the typical cell $V_o$ and thus the lines joining them to  origin are perpendicular to the corresponding boundaries.  Further, these points are also the midpoints of the chords formed by the corresponding spherical caps. This implies that these surfaces of spherical caps completely lie outside the typical cell $V_o$ (see Fig.~\ref{fig:Illustration} for $d=2$). 
Therefore, it is quite straightforward to see that the typical cell is completely contained within $\mathcal{B}_\ell(o)$ only if the set $\{L_i\}_{i=1}^{k}$ completely covers the boundary of $\mathcal{B}_\ell(o)$. Using this fact, in this section, we provide an alternate proof to the well-known spherical property of $d$-dimensional PV cells containing a large inball.

 Let the point $\tilde{\textbf{x}}_0\triangleq(R,\mytheta_0)$ denote the nearest point on the boundary of the typical cell $V_o$ to its nucleus. Therefore, $R$ is the radius of the largest ball $\mathcal{B}_{R}(o)$ contained within the typical  cell $V_o$, henceforth called the inradius of the cell. In this construction, it is evident that the nearest point $\mathbf{x}_0$ in $\Phi$ from the nucleus of $V_o$ (i.e., the origin) is at $(2R,\mytheta_0)$ such that $\|\tilde{\mathbf{x}}_0\|=\frac{1}{2}\|\mathbf{x}_0\|=R$. Note that the results presented in the following are conditioned on the inradius $R$.

Let $\mathcal{A}(r,\epsilon)$ denote the annulus formed by two balls of radii $r$ and $r+\epsilon$ co-centered at the origin. Now, consider the domain configuration $\mathcal{C}_{R}^k=\{\mathbf{\tilde{x}}_i\}_{i=1}^k$ as the set containing the mid-point of lines joining the nucleus of $V_o$ and the points in $\Phi\cap \mathcal{A}(2R,2\epsilon)$ given $\Phi(\mathcal{A}(2R,2\epsilon))=k$. Fig.~\ref{fig:PointProb_PVCell} illustrates a potential configuration of $\mathcal{C}_{R}^2$ for the case of $d=2$. By the Poisson property, the $k$ points of $\mathcal{C}_R^{k}$ are distributed uniformly at random independently of each other in the annulus $\mathcal{A}(R,\epsilon)$ such that the CDF of  $\|\tilde{\mathbf{x}}_i\|=l_i$, for $\forall i$, conditioned on $R$  is
\begin{align}
F_{l_i}(l)=\frac{l^d-R^d}{(R+\epsilon)^d-R^d},~~R\leq l\leq R+\epsilon.
\label{eq:CDF_li}
\end{align} 
\begin{figure}[t]
\centering
\includegraphics[width=.6\textwidth]{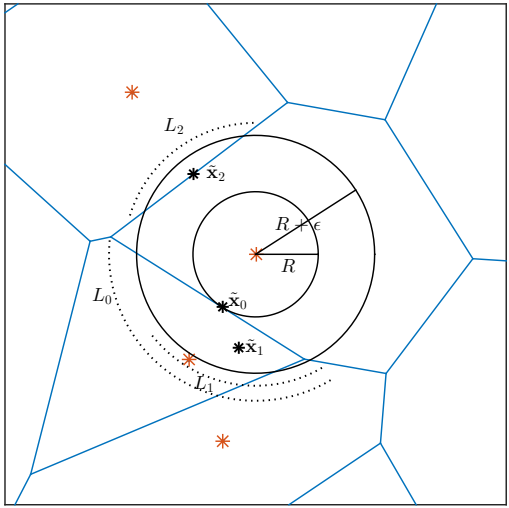}\vspace{-.7cm}
\caption{Typical cell with inradius $R$ for the case of $d=2$.}
\label{fig:PointProb_PVCell}
\end{figure}
We define the set of $k+1$ spherical caps $\{L_i\}_{i=0}^k$ corresponding to points $\{\tilde{\mathbf{x}}\}_{i=0}^k=\{ \tilde{\mathbf{x}}_0\cup\mathcal{C}_R^k\}$ on the $\mathcal{B}_{R+\epsilon}(o)$ with heights equal to $\epsilon$ for $i=0$ and $R+\epsilon-l_i$ for $i=1,\dots,k$.   
 The surface area of the spherical cap $L_i$ is \cite{li2011concise}
\begin{align}
 S_i=\begin{cases}
\frac{1}{2}\chi_d (R+\epsilon)^{d-1} I_{1-\frac{R^2}{(R+\epsilon)^2}}\left(\frac{d-1}{2},\frac{1}{2}\right),& \text{for}~i=0\\
\frac{1}{2}\chi_d (R+\epsilon)^{d-1}  I_{1-\frac{l_i^2}{(R+\epsilon)^2}}\left(\frac{d-1}{2},\frac{1}{2}\right),& \text{for}~i=1,\dots,k,  
\end{cases}
\end{align}
where $\chi_d=\frac{2\pi^{\frac{d}{2}}}{\Gamma(\frac{d}{2})}$ is the surface area of the unit radius ball in $\R^d$ and $I_z(a,b)=\frac{B_z(a,b)}{B(a,b)}$ such that $B(a,b)$ and $B_z(a,b)$ are the beta function and the incomplete beta function, respectively. Note that $0\leq S_i\leq S_0$ $\forall i$.
Since the points in $\mathcal{C}_R^k$ are i.i.d.~in $\mathcal{A}(R,\epsilon)$, the spherical caps $\{L_i\}_{i=1}^k$ of i.i.d.~surface areas are placed uniformly at random independently of each other on $\mathcal{B}_{R+\epsilon}(o)$. 
 
 Now, we evaluate the probability that the uniformly chosen point $(R+\epsilon,\myalpha)$ on the surface of $\mathcal{B}_{R_m+\epsilon}(o)$ belongs to the spherical cap $L_i$, for $i\in\{1,\dots,k\}$, as 
 \begin{align}
 p&=\mathbb{P}((R+\epsilon,\myalpha)~\text{belongs to the cap~} L_i~\text{of area}~ S_i)\nonumber\\
 &=\frac{1}{\chi_d (R+\epsilon)^{d-1}}\mathbb{E}[S_i]\nonumber\\
&\stackrel{(a)}{=}\frac{d}{2((R+\epsilon)^d-R^d)}\int_{R}^{R+\epsilon} I_{1-\frac{l^2}{(R+\epsilon)^2}}\left(\frac{d-1}{2},\frac{1}{2}\right) l^{d-1}{\rm d}l\nonumber\\
& \stackrel{(b)}=\frac{1}{2((R+\epsilon)^d-R^d)}\Bigg[ \frac{(R+\epsilon)^d}{B\left(\frac{d-1}{2},\frac{1}{2}\right)}B_{1-\frac{R^2}{(R+\epsilon)^2}}\left(\frac{d-1}{2},\frac{d+1}{2}\right)- R^d I_{1-\frac{R^2}{(R+\epsilon)^2}}\left(\frac{d-1}{2},\frac{1}{2}\right)\Bigg],
\label{eq:probability_belongs_to_cap_Li}
\end{align}  
where (a) follows using the pdf of $l_i$ which is obtained using \eqref{eq:CDF_li} and (b) follows using the steps given in Appendix \ref{app:probability_belongs_to_cap_Li}. Also note that the probability that  the uniformly chosen point $(R+\epsilon,\myalpha)$ on  the surface of $\mathcal{B}_{R+\epsilon}(o)$ belongs to the spherical cap $L_0$ is 
\begin{align}
p_0=\frac{1}{2}I_{1-\frac{R^2}{(R+\epsilon)^2}}\left(\frac{d-1}{2},\frac{1}{2}\right).
\label{eq:probability_belongs_to_cap_L0}
\end{align}

Let $K=\Phi\left(\mathcal{A}(2R,2\epsilon)\right)$. By definition, $K$ is Poisson with mean $\lambda\kappa_d((R+\epsilon)^d-R^d)$. Now to complete our argument, we evaluate the probability that the point on the boundary of $\mathcal{B}_{R+\epsilon}(o)$ does not belong to $V_o$ as  
\begin{align}
Q_d(R,\epsilon) &=\mathbb{P}((R+\epsilon,\myalpha)~\text{belongs to at least one of the caps}) \nonumber\\
&=1-(1-p_0)\mathbb{E}\left[(1-p)^K\right]\nonumber\\
&\stackrel{(a)}{=}1-\left(1-\frac{1}{2}I_{1-\frac{R^2}{(R+\epsilon)^2}}\left(\frac{d-1}{2},\frac{1}{2}\right)\right)\exp\Bigg(-\frac{1}{2}\lambda \kappa_d h(R,\epsilon)\Bigg),
\label{eq:Prob1_Out}
\end{align} 
where
\begin{equation}
h(R,\epsilon)=\frac{(R+\epsilon)^d}{B\left(\frac{d-1}{2},\frac{1}{2}\right)}B_{1-\frac{R^2}{(R+\epsilon)^2}}\left(\frac{d-1}{2},\frac{d+1}{2}\right)-R^d I_{1-\frac{R^2}{(R+\epsilon)^2}}\left(\frac{d-1}{2},\frac{1}{2}\right),
\end{equation}
and
 (a) directly follows using \eqref{eq:probability_belongs_to_cap_Li}, \eqref{eq:probability_belongs_to_cap_L0} and the probability generating function of the Poisson distribution with mean $\lambda\kappa_d((R+\epsilon)^d-R^d)$.
Now, in the following theorem we state the limiting case of \eqref{eq:Prob1_Out}.  
\begin{thm}
\label{thm:LargePVCell_Circular}
Given the inradius $R$, the probability that a point on the boundary of $\mathcal{B}_{R+\epsilon}(o)$ does not belong to the PV cell $V_o$ approaches one  as $R$ tends to infinity, i.e.,
\begin{equation}
\lim_{R\to\infty}Q_d(R,\epsilon)=1, ~~~~~~~~\forall \epsilon>0.
\label{eq:LargePVCell_Circular} 
\end{equation}
\end{thm}
\begin{proof}
We note that, for $\epsilon>0$, $I_{1-\frac{R^2}{(R+\epsilon)^2}}\left(\frac{d-1}{2},\frac{1}{2}\right)\to 0$ as $R\to\infty$. Therefore, in order to prove \eqref{eq:LargePVCell_Circular}, it is sufficient to show that the exponential term in \eqref{eq:Prob1_Out} tends to 0 as $R\to\infty$ for $\epsilon>0$, i.e.,
\begin{align}
\lim_{R\to\infty}{h}(R,\epsilon)=\infty.\nonumber
\end{align}
To this end, we multiply ${h}(R,\epsilon)$ with $B\left(\frac{d-1}{2},\frac{1}{2}\right)$ to obtain
\begin{align}
\tilde{h}(R,\epsilon)=(R+\epsilon)^d B_{1-\frac{R^2}{(R+\epsilon)^2}}\left(\frac{d-1}{2},\frac{d+1}{2}\right)-R^d B_{1-\frac{R^2}{(R+\epsilon)^2}}\left(\frac{d-1}{2},\frac{1}{2}\right).
\label{eq:Converegence_show}
\end{align}
We have 
\begin{equation}
 B_{1-\frac{R^2}{(R+\epsilon)^2}}\left(\frac{d-1}{2},a\right)=\int_0^{1-\frac{R^2}{(R+\epsilon)^2}}t^{\frac{d-1}{2}-1}(1-t)^{a-1} {\rm d}t.\nonumber
\end{equation}
Thus, using the binomial expansion of the term $(1-t)^{a-1}$, we get
\begin{align*}
B_{1-\frac{R^2}{(R+\epsilon)^2}}\left(\frac{d-1}{2},a\right)&=\int_0^{1-\frac{R^2}{(R+\epsilon)^2}}t^{\frac{d-1}{2}-1}\sum_{k=0}^{\infty} (-1)^k \frac{1}{k!}\prod_{l=0}^{k-1}(a-1-l)  t^k {\rm d}t\\
&=\sum_{k=0}^{\infty}(-1)^k \frac{1}{k!}\prod_{l=0}^{k-1}(a-1-l) \int_0^{1-\frac{R^2}{(R+\epsilon)^2}}t^{k+\frac{d-1}{2}-1}  {\rm d}t\\
&=\sum_{k=0}^{\infty}(-1)^k \frac{\prod_{l=0}^{k-1}(a-1-l)}{k!\left(k+\frac{d-1}{2}\right)}  \left(1-\frac{R^2}{(R+\epsilon)^2}\right)^{k+\frac{d-1}{2}}.  
\end{align*}
Let $A_k=\frac{1}{k!(k+\frac{d-1}{2})}\prod_{l=0}^{k-1}\left(\frac{d+1}{2}-1-l\right)$ and $B_k=\frac{1}{k!(k+\frac{d-1}{2})}\prod_{l=0}^{k-1}\left(\frac{1}{2}-1-l\right)$.  
Using the above series expansion of the incomplete beta function, we can rewrite \eqref{eq:Converegence_show} as
\begin{align*}
\tilde{h}(R,\epsilon)&=\sum_{k=0}^{\infty}(-1)^k \left[A_k(R+\epsilon)^d-B_kR^d\right]\left(1-\frac{R^2}{(R+\epsilon)^2}\right)^{k+\frac{d-1}{2}}\nonumber\\
&=\sum_{k=0}^{\infty}(-1)^k \left[(A_k-B_k)R^d + A_k\sum_{n=0}^{d-1} {d\choose n} R^n\epsilon^{d-n}\right]\frac{(2R\epsilon+\epsilon^2)^{k+\frac{d-1}{2}}}{(R+\epsilon)^{2k+d-1}}\nonumber\\
&=\sum_{k=0}^{\infty}(-1)^k \left[(A_k-B_k)R^{\frac{d+1}{2}-k} + A_k\sum_{n=0}^{d-1} {d\choose n} R^{n+\frac{1-d}{2}-k}\epsilon^{d-n}\right]\frac{(2\epsilon+R^{-1}\epsilon^2)^{k+\frac{d-1}{2}}}{(1+R^{-1}\epsilon)^{2k+d-1}}.\nonumber
\end{align*}
Now note that $A_k-B_k\geq 0$ for $k\leq\frac{d-1}{2}$.
Therefore, the terms in the above summation tend to infinity as $R$ tends to infinity for $k<\frac{d+1}{2}$. In addition, the terms converge to a constant for $k=\frac{d+1}{2}$ (if $d$ is odd) and to zero for $k>\frac{d+1}{2}$.  From this, it is clear that $\tilde{h}(R,\epsilon)\to\infty$  as $R\to\infty$. Therefore, we have ${h}(R,\epsilon)\to\infty$ as $R\to\infty$. 
\end{proof}
From Theorem \ref{thm:LargePVCell_Circular}, it is easy to see that the boundary of a PV cell $V_o$ must be contained within the annulus  $\mathcal{A}(R,\epsilon)$ as its inradius $R\to\infty$ for an arbitrarily small $\epsilon$. Hence PV cells with large inradii tend to be spherical. Therefore, the approach presented in this section provides an alternate proof for the well-known {\em spherical} nature of the PV cells having a large inball \cite{calka2005limit,Calka2002,miles1995heuristic}.  A realization of a PV cell $V_o$ with large inradius is shown in Fig.~\ref{fig:Illustration_largPVCell} for the case of $d=2$.                                                                                                                                                                                                                                                                                                            
\begin{figure}[h]
\centering
\includegraphics[width=.55\textwidth]{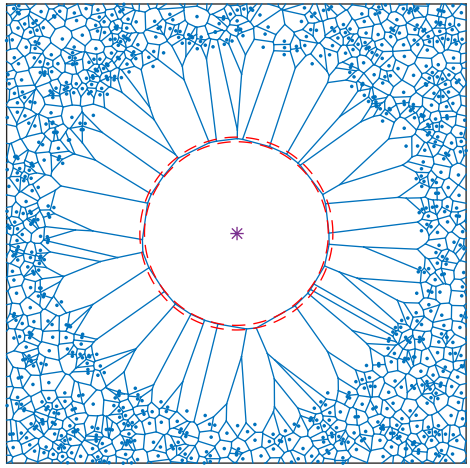}\vspace{-.6cm}
\caption{Illustration of a cell in $\mathbb{R}^2$ with large inradius.}
\label{fig:Illustration_largPVCell}
\end{figure}
\appendices
\section{Solution of Integrals in \eqref{eq:CDF_Ro_d1_Integrals}}
\label{app:CDF_Ro_d1_Integrals}
We have
\begin{align*}
F_{R_o}(r) = &\underbrace{\int_0^r\int_0^{r_2}8\lambda^2\exp(-2\lambda(r_1+r_2)){\rm d}r_1{\rm d}r_2}_{\mathtt{Int}_1}  +\underbrace{\int_r^\infty\int_0^{r}\frac{r+r_1}{r_1+r_2}8\lambda^2\exp(-2\lambda(r_1+r_2)){\rm d}r_1{\rm d}r_2}_{\mathtt{Int}_2} \\
& +\underbrace{\int_r^\infty\int_r^{r_2}\frac{2r}{r_1+r_2}8\lambda^2\exp(-2\lambda(r_1+r_2)){\rm d}r_1{\rm d}r_2}_{\mathtt{Int}_3}. \numberthis
\label{eq:CDF_Ro_d1_Integrals1}
\end{align*}
First of all, it is easy to show that $\mathtt{Int}_1$ reduces to 
\begin{equation}
\mathtt{Int}_1=1+\exp(-4\lambda r)-2\exp(-2\lambda r).\label{eq:I_1}
\end{equation}
Now, we have 
\begin{align*}
\mathtt{Int}_2 = & \underbrace{\int_r^\infty\int_0^{r}\frac{r}{r_1+r_2}8\lambda^2\exp(-2\lambda(r_1+r_2)){\rm d}r_1{\rm d}r_2}_{\mathtt{Int}_{21}}  +\underbrace{\int_r^\infty\int_0^{r}\frac{r_1}{r_1+r_2}8\lambda^2\exp(-2\lambda(r_1+r_2)){\rm d}r_1{\rm d}r_2}_{\mathtt{Int}_{22}}.
\end{align*}
By substituting $r_1+r_2=y$, we solve $\mathtt{Int}_{21}$ as 
\begin{align*}
\mathtt{Int}_{21}&=8\lambda^2 r\int_r^\infty\int_{r_2}^{r+r_2}\frac{1}{y}\exp(-2\lambda y){\rm d}y{\rm d}r_2\\
&=8\lambda^2 r\int_r^\infty\int_{r_2}^{\infty}\frac{1}{y}\exp(-2\lambda y){\rm d}y{\rm d}r_2-8\lambda^2 r\int_r^\infty\int_{r+r_2}^{\infty}\frac{1}{y}\exp(-2\lambda y){\rm d}y{\rm d}r_2\\
&=8\lambda^2 r\int_r^\infty\int_{2\lambda r_2}^{\infty}\frac{1}{z}\exp(-z){\rm d}z{\rm d}r_2-8\lambda^2 r\int_r^\infty\int_{2\lambda (r+r_2)}^{\infty}\frac{1}{z}\exp(-z){\rm d}z{\rm d}r_2\\
&=8\lambda^2r\int_r^\infty\mathtt{E}_1(2\lambda r_2){\rm d}r_2-8\lambda^2 r\int_{r}^\infty\mathtt{E}_1(2\lambda(r+r_2)){\rm d}r_2\\
&=8\lambda^2r\int_r^\infty\mathtt{E}_1(2\lambda r_2){\rm d}r_2-8\lambda^2 r\int_{2 r}^\infty\mathtt{E}_1(2\lambda u){\rm d}u,
\end{align*}
where $\mathtt{E}_1$ is an exponential integral function. 
From \cite[Eq. 5.22.8]{gradshteyn2014table}, we have 
\begin{equation}
\int_x^\infty\mathtt{E}_1(az){\rm d}z=\frac{1}{a}\exp(-ax)-x\mathtt{E}_1(ax).\label{eq:IntegralofE1}
\end{equation}
Therefore, we get
\begin{align*}
\mathtt{Int}_{21}=4\lambda r\exp(-2\lambda r)-4\lambda r\exp(-4\lambda r)-8\lambda^2 r^2\mathtt{E}_1(2\lambda r)+16\lambda^2 r^2\mathtt{E}_1(4\lambda r).\numberthis
\label{eq:I_21}
\end{align*}
Similarly, by substituting $r_1+r_2=y$, we solve $\mathtt{Int}_{22}$ as 
\begin{align*}
\mathtt{Int}_{22}&=8\lambda^2\int_r^\infty\int_{r_2}^{r+r_2}\frac{y-r_2}{y}\exp(-2\lambda y){\rm d}y{\rm d}r_2\\
&=8\lambda^2 \int_r^\infty \left(\int_{r_2}^{r+r_2}\exp(-2\lambda y){\rm d}y-r_2\int_{r_2}^{r+r_2}\frac{1}{y}\exp(-2\lambda y){\rm d}y\right){\rm d}r_2\\
&=8\lambda^2 \int_r^\infty \frac{1}{2\lambda}\left(\exp(-2\lambda r_2)-\exp(-2\lambda(r+r_2))\right)-r_2\left(\mathtt{E}_1(2\lambda r_2)-\mathtt{E}_1(2\lambda(r+r_2))\right){\rm d}r_2\\
&=2\left(1-\exp(-2\lambda r)\right)\exp(-2\lambda r)-8\lambda^2 \underbrace{\int_r^\infty r_2\mathtt{E}_1(2\lambda r_2){\rm d}r_2}_{\mathtt{Int}_{221}}+8\lambda^2\underbrace{\int_r^\infty r_2\mathtt{E}_1(2\lambda(r+r_2)){\rm d}r_2}_{\mathtt{Int}_{222}}.\numberthis
\label{eq:I_22}
\end{align*}
Now, using \eqref{eq:IntegralofE1} and the integration by parts, we solve $\mathtt{Int}_{221}$ as 
\begin{align*}
\mathtt{Int}_{221}&=\int_r^\infty r_2\mathtt{E}_1(2\lambda r_2){\rm d}r_2\\
&=r_2\int\mathtt{E}_1(2\lambda r_2){\rm d}r_2\mid_r^\infty - \int_r^\infty\left(\int\mathtt{E}_1(2\lambda r_2){\rm d}r_2\right){\rm d}r_2\\
&=\frac{r}{2\lambda}\exp(-2\lambda r)-r^2\mathtt{E}_1(2\lambda r) - \int_r^\infty \left(r_2\mathtt{E}_1(2\lambda r_2)-\frac{1}{2\lambda}\exp(-2\lambda r_2)\right){\rm d}r_2\\
&=\frac{r}{2\lambda}\exp(-2\lambda r)-r^2\mathtt{E}_1(2\lambda r) +\frac{1}{4\lambda^2 }\exp(-2\lambda r)-\mathtt{Int}_{221}\\
&=\frac{r}{4\lambda}\exp(-2\lambda r)+\frac{1}{8\lambda^2}\exp(-2\lambda r)-\frac{r^2}{2}\mathtt{E}_1(2\lambda r).\numberthis\label{eq:I_221}
\end{align*}
Now, 
\begin{align*}
\mathtt{Int}_{222}&=\int_r^\infty r_2\mathtt{E}_1(2\lambda(r+r_2)){\rm d}r_2\\
&=\frac{1}{4\lambda^2 }\int_{4\lambda r}^\infty(y-2\lambda r)\mathtt{E}_1(y){\rm d}y\\
&=\underbrace{\frac{1}{4\lambda^2 }\int_{4\lambda r}^\infty y\mathtt{E}_1(y){\rm d}y}_{A}-\underbrace{\frac{r}{2\lambda }\int_{4\lambda r}^\infty\mathtt{E}_1(y){\rm d}y}_{B}\\
&=\frac{1}{4\lambda^2}y\int\mathtt{E}_1(y){\rm d}y\mid_{4\lambda r}^\infty-\frac{1}{4\lambda^2}\int_{4\lambda 
r}^{\infty}\left(\int\mathtt{E}_1(y){\rm d}y\right){\rm d}y-B\\
&=\frac{r}{\lambda}\exp(-4\lambda r)-4r^2\mathtt{E}_1(4\lambda r)-\frac{1}{4\lambda^2}\int_{4\lambda r}^\infty\left(y\mathtt{E}_1(y)-\exp(-y)\right){\rm d}y-B\\
&=\frac{r}{\lambda}\exp(-4\lambda r)-4r^2\mathtt{E}_1(4\lambda r)+\frac{1}{4\lambda^2}\exp(-4\lambda r)-A-B\\
&\stackrel{(a)}{=}\frac{r}{2\lambda}\exp(-4\lambda r)+\frac{1}{8\lambda^2}\exp(-4\lambda r)-2r^2\mathtt{E}_1(4\lambda r)-\frac{r}{2\lambda}\left(\exp(-4\lambda r)-4\lambda r\mathtt{E}_1(4\lambda r)\right)\\
&=\frac{1}{8\lambda^2}\exp(-4\lambda r),
\numberthis\label{eq:I_222}
\end{align*}
where step (a) follows by substituting $A+B = \mathtt{Int}_{222}+2B$ and $B=\frac{r}{2\lambda}(\exp(-4\lambda r)-4\lambda r\mathtt{E}_1(4\lambda r))$.
Substituting \eqref{eq:I_221} and \eqref{eq:I_222} in \eqref{eq:I_22}, we get
\begin{align*}
\mathtt{Int}_{22}&=\exp(-2\lambda r)-2\lambda r\exp(-2\lambda r)-\exp(-4\lambda r)+4\lambda^2 r^2\mathtt{E}_1(2\lambda r).\numberthis
\label{eq:I_22a}
\end{align*}
Now, adding \eqref{eq:I_21} and \eqref{eq:I_22a}, we get
\begin{align*}
\mathtt{Int}_2&=\exp(-2\lambda r)+2\lambda  r\exp(-2\lambda r)-\exp(-4\lambda r)-4\lambda r\exp(-4\lambda r)-4\lambda^2 r^2\mathtt{E}_1(2\lambda r)+16\lambda^2 r^2\mathtt{E}_1(4\lambda r).\numberthis
\label{eq:I_2}
\end{align*}
Again substituting $r_1+r_2=y$ and using \eqref{eq:IntegralofE1}, we evaluate $\mathtt{Int}_3$ as 
\begin{align*}
\mathtt{Int}_3&=16\lambda^2 r\int_r^\infty\int_{r+r_2}^{2r_2}\frac{1}{y}\exp(-2\lambda y){\rm d}y{\rm d}r_2\\
&=16\lambda^2 r\int_r^\infty\mathtt{E}_1(2\lambda (r+r_2)){\rm d}r_2-16\lambda^2 r\int_r^\infty\mathtt{E}_1(4\lambda r_2){\rm d}r_2\\
&=8\lambda r\int_{4\lambda r}^\infty\mathtt{E}_1(u){\rm d}u-16\lambda^2 r \left(\frac{1}{4\lambda}\exp(-4\lambda r)-r\mathtt{E}_1(4r)\right)\\
&=8\lambda r\left(\exp(-4\lambda r)-4\lambda r\mathtt{E}_1(4\lambda r)\right)-4\lambda r\exp(-4\lambda r)+16\lambda^2r^2\mathtt{E}_1(4r)\\
&=4\lambda r\exp(-4\lambda r)-16\lambda^2r^2\mathtt{E}_1(4\lambda r).\numberthis\label{eq:I_3}
\end{align*}
Finally, adding \eqref{eq:I_1}, \eqref{eq:I_2} and \eqref{eq:I_3}, we get \eqref{eq:CDF_R_1D}.
\section{Solution of Integral in \eqref{eq:probability_belongs_to_cap_Li}}
\label{app:probability_belongs_to_cap_Li}
Let $a=\frac{d-1}{2}$ and $b=\frac{1}{2}$. From step (a) of \eqref{eq:probability_belongs_to_cap_Li} and using $I_x(a,b)=\frac{B_x(a,b)}{B(a,b)}$, we have
\begin{align*}
p&=\nu_R\int_{R}^{R+\epsilon} B_{{1-\frac{l^2}{(R+\epsilon)^2}}}\left(a,b\right) l^{d-1}{\rm d}l,
\end{align*}
where $\nu_R=\frac{{d((R+\epsilon)^d-R^d)^{-1}}}{2B\left(a,b\right)}$. 
We solve the above integral using integration by parts as follows.
Let $v=l^{d-1}$ and $u=B_{1-\frac{l^2}{(R+\epsilon)^2}}\left(\frac{d-1}{2},\frac{1}{2}\right)$. We have $$\frac{{\rm d}}{{\rm d}l}u=-\frac{2l}{(R+\epsilon)^2}\left(\frac{l^{2}}{(R+\epsilon)^{2)}}\right)^{b-1}\left(1-\frac{l^2}{(R+\epsilon)^2}\right)^{a-1},$$ and thus
\begin{align*}
p&=-\nu_R\frac{R^d}{d}B_{1-\frac{R^2}{(R+\epsilon)^2}}\left(a,b\right) + \nu_R\int_R^{R+\epsilon} \frac{l^d}{d} \frac{2l}{(R+\epsilon)^2}\frac{l^{2(b-1)}}{(R+\epsilon)^{2(b-1)}}\left(1-\frac{l^2}{(R+\epsilon)^2}\right)^{a-1}{\rm d}l.
\end{align*} 
Now, substituting $\frac{l^2}{(R+\epsilon)^2}=z$, we get
\begin{align*}
p&=-\nu_R\frac{R^d}{d}B_{1-\frac{R^2}{(R+\epsilon)^2}}\left(a,b\right)+\nu_R\frac{(R+\epsilon)^d}{d}\int_\frac{R^2}{(R+\epsilon)^2}^1z^{b+\frac{d}{2}-1}(1-z)^{a-1}{\rm d}z\\
&=-\nu_R\frac{R^d}{d}B_{1-\frac{R^2}{(R+\epsilon)^2}}\left(a,b\right)+ \nu_R\frac{(R+\epsilon)^d}{d}\left[B\left(b+\frac{d}{2}-1,a\right) - B_{\frac{R^2}{(R+\epsilon)^2}}\left(b+\frac{d}{2}-1,a\right)\right]\\
&=-\tilde{\nu}_R R^d I_{1-\frac{R^2}{(R+\epsilon)^2}}(a,b) +\tilde{\nu}_R(R+\epsilon)^d\frac{B\left(b+\frac{d}{2},a\right)}{B(a,b)}\left[1-I_{1-\frac{R^2}{(R+\epsilon)^2}}\left(b+\frac{d}{2},a\right)\right]\\
&=-\tilde{\nu}_R R^d I_{1-\frac{R^2}{(R+\epsilon)^2}}(a,b) +\tilde{\nu}_R(R+\epsilon)^d\frac{B\left(b+\frac{d}{2},a\right)}{B(a,b)}\left[1-I_{1-\frac{R^2}{(R+\epsilon)^2}}\left(b+\frac{d}{2},a\right)\right]\\
&=\tilde{\nu}_R(R+\epsilon)^d B_{1-\frac{R^2}{(R+\epsilon)^2}}\left(a,b+\frac{d}{2}\right)-\tilde{\nu}_R R^d I_{1-\frac{R^2}{(R+\epsilon)^2}}(a,b) 
\end{align*}
where $\tilde{\nu}_R=\frac{1}{2((R+\epsilon)^d-R^d)}$. The last equality follows using $I_{x}(a,b)=1-I_{1-x}(b,a)$ and $B(a,b)=B(b,a)$.


%
%
%
%

\end{document}